\documentclass[reqno]{amsart}
\usepackage{amsfonts,amsmath,amsthm,amssymb,latexsym, cite}%
\usepackage{graphicx,verbatim, dsfont}
\usepackage[colorlinks,plainpages,citecolor=blue, linkcolor=blue]{hyperref}%,bookmarksnumbered,,backref
\theoremstyle{plain}
\newtheorem{thm}{Theorem}

\newtheorem{lem}{Lemma}
\newtheorem{rem}{Remark}
\newtheorem{prop}{Proposition}

\numberwithin{equation}{section}

\newcommand{\dom}{\mathop{\rm dom}}

\newcommand{\supp}{\mathop{\rm supp}}

\renewcommand{\kappa}{\varkappa}

\newcommand{\Real}{\mathbb R}
\newcommand{\Cmpl}{\mathbb C}
\newcommand{\eps}{\varepsilon}

\newcommand{\cI}{\mathcal{I}}

\newcommand{\cE}{\mathcal{E}}

\newcommand{\lme}{\lambda^\eps}

\newcommand{\cH}{\mathcal{H}}

\newcommand\nep{\textstyle\frac r\eps}

\newcommand{\pte}{\partial_n}

\newcommand{\npt}{\partial_\nu}

\begin{document}

\title[2D Schr\"{o}dinger operators with singular potentials]
{2D Schr\"{o}dinger operators with singular potentials concentrated near  curves}

\author{Yuriy Golovaty}%
\address{Department of Mechanics and Mathematics,
  Ivan Franko National University of Lviv\\
  1 Universytetska str., 79000 Lviv, Ukraine}
\email{yuriy.golovaty@lnu.edu.ua}


\begin{abstract}
We investigate the Schr\"{o}dinger operators $H_\eps=-\Delta +W+V_\eps$ in $\Real^2$
with the short-range potentials $V_\eps$ which are localized around a smooth closed curve $\gamma$. The operators $H_\eps$ can be viewed as an approximation of the heuristic Hamiltonian $H=-\Delta+W+a\partial_\nu\delta_\gamma+b\delta_\gamma$, where  $\delta_\gamma$ is  Dirac's $\delta$-function supported on $\gamma$ and $\partial_\nu\delta_\gamma$ is its normal derivative on $\gamma$. Assuming that the  operator $-\Delta +W$ has only discrete spectrum, we analyze the asymptotic behaviour of eigenvalues and eigenfunctions of $H_\eps$. The transmission conditions on $\gamma$ for the eigenfunctions
$ u^+=\alpha u^-$, $\alpha\, \partial_\nu u^+-\partial_\nu u^-=\beta u^-$,
which arise in the limit as $\eps\to 0$, reveal a nontrivial connection between spectral properties of $H_\eps$ and the geometry of $\gamma$.
\end{abstract}

\maketitle

%%%%%%%%%%%%%%%%%%%%%%%%%%%%%%%%%%%%%%%%%%%%%%%%%%%%%%%%%%%%%%%%%%%%%%%%%%
% Introduction
%%%%%%%%%%%%%%%%%%%%%%%%%%%%%%%%%%%%%%%%%%%%%%%%%%%%%%%%%%%%%%%%%%%%%%%%%%

\section{Introduction}

Solvable type operators with interactions supported by manifolds of a lower dimension have attracted considerable attention both in the physical and mathematical literature in recent years. Such  operators are of interest in applications of mathematics in different fields of science and engineering because they reveal  unquestioned effectiveness whenever the exact solvabi\-li\-ty together with a nontrivial description of an actual physical phenomenon is required. The Schr\"{o}dinger ope\-rators with pseudo-potentials  that are distributions supported by curves, surfaces, metric graphs  are used successfully for modelling  quantum systems  with charged inclusions,  leaky quantum graphs, quantum waveguides etc. There exists a large body of results on this subject, but we conﬁne ourselves to the case of hypersurfaces (i.e., manifolds of codimension one) as interaction supports.  This case is a natural generalization to higher dimensions of the one-dimensional Hamiltonians with point interactions such as $\delta$- or $\delta'$-interactions.

The Schr\"{o}dinger operators  formally written as
\begin{equation}\label{FormalDelta}
  -\Delta+\alpha \delta_S
\end{equation}
with  $\delta$ potentials supported by  compact or non-compact orientable  hypersurfaces~$S$   have attracted special attention in the last decade.
The pseudo-potential $\alpha \delta_S$ is a distribution in $\mathcal{D}'(\Real^n)$ acting as
\begin{equation*}
  \langle\alpha \delta_S, \phi\rangle=\int_S \alpha\phi\,d\sigma
\end{equation*}
for any $\phi\in C^\infty_0(\Real^n)$, where $\alpha$ is a locally integrable function on $S$ and  $\sigma$ is the natural measure on $S$ induced by the Riemannian
metric. The physical motivation comes  from nuclear, molecular and solid-state physics \cite{GreenMoszkowski1965, Lloyd1965, FaesslerPlastino1967, Blinder1978}, where the so-called SDI model (surface delta interaction) has been used since 1965.

If the hypersurface $S$ is smooth enough, one can give meaning to the heuristic expression \eqref{FormalDelta} in different ways. We can suppose that $H_{\alpha, S}$ is the Laplacian acting on the functions $f\in W_2^2(\Real^n\setminus S)$ satisfying the transmission conditions
\begin{equation}\label{CndsDeltaOnS}
  f^+=f^-, \quad \partial_\nu f^+-\partial_\nu f^-=\alpha f \quad\text{on } S.
\end{equation}
Here $f^-$ and $f^+$ denote the one-side traces of $f$ on $S$ and $\nu$ is the normal vector field on $S$.
The formal expression \eqref{FormalDelta} can also be defined rigorously via the symmetric sesquilinear form
\begin{equation*}
  \mathfrak{a}(f,g)=(\nabla f,\nabla g)_{L_2(\Real^n;\Cmpl^n)}+\int_S \alpha f|_S\: \bar{g}|_S\,d\sigma, \qquad \dom \mathfrak{a}=W_2^1(\Real^n),
\end{equation*}
where $f|_S$ denote the trace of function $f\in W_2^1(\Real^n)$ on $S$. The form $\mathfrak{a}$ is a densely defined, closed, and semibounded  in $L_2(\Real^n)$, and hence there
exists a  self-adjoint operator $A_{\alpha, S}$ in $L_2(\Real^n)$ such that $(A_{\alpha, S}f,g)_{L_2(\Real^n)}=\mathfrak{a}(f,g)$ for all $f\in \dom A_{\alpha, S}$ and $g\in \dom \mathfrak{a}$.

An essential  advantage of this model is its ``stability'' with respect to regula\-ri\-zations by the Schr\"{o}dinger operators with short-range potentials.
If a family of potentials $U_\eps$ with compact supports converges to $\alpha \delta_S$ in the space of distributions, then the operators $-\Delta+U_\eps$
converge to $H_{\alpha, S}$ in the norm resolvent sense, as $\eps\to 0$ \cite{AntoineGesztesyShabani1987, Shimada1992, BehrndtExnerHolzmannLotoreichik2017}.
The self-adjointness, spectral properties,  scattering behaviour of the Schr\"{o}dinger operators with $\delta_S$-interactions were investigated in numerous articles in the recent past; we mention here
\cite{AntoineGesztesyShabani1987, Shimada1992, Shimada1994-1, Shimada1994-2, ExnerFraas2007, ExnerFraas2008, AlbeverioKostenkoMalamudNeidhardt2014}
for interactions on finite or infinite families of  concentric spheres, \cite{BehrndtExnerHolzmannLotoreichik2017, ExnerFraas2009,  BehrndtLangerLotoreichik2013, BehrndtExnerLotoreichik2014-1, ExnerJex2014, LotoreichikRohleder2015, DittrichExnerKuhn2016, MantilePosilicanoSini2016, Lotoreichik2019}  on  closed hypersurfaces and hypersurfaces with boundary, and \cite{BehrndtExnerLotoreichik2014-2, OurmieresBonafosPankrashkin2018} for interactions concentrated near conical surfaces.

Similarly, one can also generalize to higher dimensions the four-parameter fa\-mi\-ly of singular point interactions on the line, see \cite{ExnerRohleder2016}. However, for some  reason a very popular interaction in the  literature, of course, in addition to the  $\delta_S$-interaction, is the so-called $\delta'$-interaction supported on hypersurfaces \cite{LotoreichikRohleder2015, ExnerKhrabustovskyi2015, Jex2015, JexLotoreichik2016}. This interaction is characterized by the transmission
 conditions
\begin{equation*}
\partial_\nu f^+=\partial_\nu f^-, \quad f^+- f^-=\alpha \partial_\nu f\quad \text{on } S.
\end{equation*}

Despite all advantages of the solvable models, they  give rise to many ma\-the\-matical difficulties. One of them deals with the  multiplication of distributions;
many Schr\"{o}dinger operators with singular potentials are often only formal expressions without a precise or unambiguous mathematical meaning.
The aim of the present paper is to find proper solvable models, i.e., proper transmission conditions on an interaction support, for the pseudo-Hamiltonians
 \begin{equation}\label{FormalDeltaPrime}
 H=-\Delta+W+a\partial_\nu\delta_\gamma+b\delta_\gamma
\end{equation}
in $\Real^2$, where $W$ is a regular potential, $a$ and $b$ are some functions on the closed curve $\gamma$, and $\delta_\gamma$ is  Dirac's $\delta$-function supported on $\gamma$.
 The pseudo-potential $a\partial_\nu\delta_\gamma+b \delta_\gamma$ is a distribution in $\mathcal{D}'(\Real^2)$ acting as
\begin{equation*}
  \langle a\partial_\nu\delta_\gamma+b \delta_\gamma, \phi \rangle= \int_\gamma  (-\npt(a\phi)+ b \phi)\,d\gamma
\end{equation*}
for all test functions $\phi\in C^\infty_0(\Real^2)$. It should be noted that this problem is not related to the above-mentioned $\delta'$-interactions.

In contrast to \eqref{FormalDelta},  heuristic expression \eqref{FormalDeltaPrime} is more singular, and the problem of giving its strictly mathematical meaning is more subtle.
First of all,  $H$ is ``unstable'' under regularizations by Hamiltonians $H_\eps=-\Delta +W+V_\eps$ with short-range potentials $V_\eps$. From a physical viewpoint, it means that the quantum systems with  $(a\partial_\nu\delta_\gamma+b\delta_\gamma)$-like localized potentials of different  shapes possess slightly different properties. Our purpose is to find solvable mo\-dels describing with admissible fidelity the real quantum processes  governed by  $H_\eps$ for given  $V_\eps$  suitably scaled in the normal direction to $\gamma$. The mathematical motivation for studying such operators is also that they exhibit nontrivial  relations between some spectral
properties and the geo\-metry of $\gamma$ which arise (upon passing to the limit) in the family of transmission conditions
\begin{equation*}
  f^+=\alpha f^-, \quad \alpha\, \partial_\nu f^+-\partial_\nu f^-=\beta f^- \quad\text{on } \gamma.
\end{equation*}
Assuming that the unperturbed operator $-\Delta+W$ has only  point spectrum, we analyze the asymptotic behaviour of eigenvalues and eigenfunctions of $H_\eps$ as $\eps\to 0$.

The one-dimensional case of the problem was studied in
\cite{GolovatyManko2009, GolovatyHrynivJPA2010, Golovaty2012, GolovatyHrynivProcEdinburgh2013, GolovatyIEOT2013}.
For the Schr\"{o}\-din\-ger operators with $(\alpha\delta'+\beta\delta)$-like potentials,  the norm resolvent convergence  was established
and a family of exactly solvable models was  obtained (see also \cite{ChristianZolotarIermak03, Zolotaryuk08, Zolotaryuk09} for regularizations of $\alpha\delta'+\beta\delta$ by piecewise constant potentials).

\begin{figure}[hb]
\centering
\includegraphics[scale=.45]{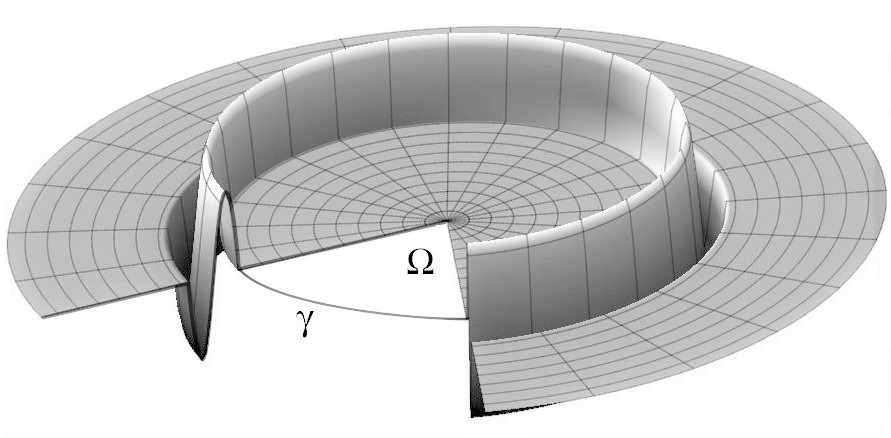}
\caption{The $\partial_\nu\delta_\gamma$-like potential (part of the plot is cut out for better visualization).}
\label{FigPdDeltaLike}
\end{figure}

The distribution $-\partial_\nu\delta_\gamma$ can be interpreted as the Laplacian of the indicator (or the characteristic function) $\mathds{1}_{x\in\Omega}$ of the bounded domain $\Omega$ enclosed by $\gamma$. Namely,
\begin{equation*}
  \partial_\nu\delta_\gamma=-\Delta \mathds{1}_{x\in\Omega}
\end{equation*}
in the sense of distributions. Suppose that $\chi^\eps$ is a sequence of smooth  functions which are identically one on  $\Omega$,  vanish arbitrarily close to $\Omega$, and
$\chi^\eps\to \mathds{1}_{x\in\Omega}$ in $L^1(\Real^2)$ as $\eps\to 0$. Then $V_\eps=-\Delta \chi^\eps$ is an example of $\partial_\nu\delta_\gamma$-like potentials, see Fig.~\ref{FigPdDeltaLike}. Obviously, $V_\eps$ converges to $\partial_\nu\delta_\gamma$ in $\mathcal{D}'(\Real^2)$ as $\eps\to 0$.
Such potentials are interesting not only in the context of Schr\"{o}dinger operators, they also arise in the Navier-Stokes equations,   free boundary problems, and in the potential theory for parabolic and elliptic PDE in bounded domains.
The Laplacian of the indicator function with its grid adaptations is the base of the front-tracking method. This numerical method  allows simulating  unsteady multi-fluid flows  in which a sharp interface  separates incompressible fluids of different density and viscosity \cite{UnverdiTryggvason1992}, a time dependent two-dimensional dendritic solidification of pure substances \cite{JuricTryggvason1996},   flow-flexible body interactions with large deformation \cite{UddinSung2012}.
The Laplacian of the indicator and its regularizations have been used  to establish some relationships between the Dirichlet and Neumann boundary value problems for the heat and Laplace equations  and the theory  of the Feynman path integrals \cite{Lange2012}.

%%%%%%%%%%%%%%%%%%%%%%%%%%%%%%%%%%%%%%%%%%%%%%%%%%%%%%%%%%%%%%%%%%%%%%%%%%
% Statement of Problem and Main Results
%%%%%%%%%%%%%%%%%%%%%%%%%%%%%%%%%%%%%%%%%%%%%%%%%%%%%%%%%%%%%%%%%%%%%%%%%%

\section{Statement of Problem and Main Results}

We study the family of   Schr\"{o}dinger operator
\begin{equation*}
H_\eps=-\Delta +W+V_\eps
\end{equation*}
in $L^2(\Real^2)$, where the potential $W$ belongs to $L^\infty_{loc}(\Real^2)$  and increases as $|x|\to +\infty$.
Let $\gamma$ be a  closed  smooth curve in $\Real^2$ without self-intersection points. Assume also that $W$ is smooth in a neighbourhood of $\gamma$.
We define the short-range potentials $V_\eps$ as follows.  Let $\omega_\eps$ be the $\eps$-neighborhood of $\gamma$, i.e., the union of all open balls of radius $\eps$ around a point on $\gamma$. For $\eps$ small enough, $\omega_\eps$ is a domain with smooth boundary. To specify   explicit dependence of $V_\eps$ on  $\eps$ we introduce  local coordinates  in $\omega_\eps$ (see Fig.~\ref{FigLC}).
Let $\alpha\colon [0,|\gamma|)\to \Real^2$ be the unit-speed smooth parametrization of $\gamma$ with the natural parameter $s$, and $|\gamma|$ is the length of $\gamma$. Then the vector $\nu=(-\dot{\alpha}_2, \dot{\alpha}_1)$ is a unit normal on $\gamma$. Set $ x=\alpha(s)+r\nu(s)$ for $(s,r)\in [0,|\gamma|)\times (-\eps, \eps)$,
where $r$ is the signed distance from $x$ to $\gamma$.
Suppose that  $V_\eps$ has the form
\begin{equation}\label{Veps}
V_\eps\big(\alpha(s)+r\nu(s)\big)=\eps^{-2}\,V\left(\eps^{-1}r\right)
+\eps^{-1}\,U\left(s,\eps^{-1}r\right),
\end{equation}
where $V$ and $U$ are smooth functions such  that the supports
of $V$ and $U(s,\,\cdot\,)$ lie in the interval $[-1,1]$ for all $s$. Hence, $\supp V_\eps\subset \omega_\eps$.
In general, the potentials $V_\eps$ diverge in the space of distributions $\mathcal{D}(\Real^2)$;
 $V_\eps$ converge only if $V$ is a zero mean function, as we will show below. In this case, $V_\eps\to a\partial_\nu\delta_\gamma+b\delta_\gamma$ in $\mathcal{D}'(\Real^2)$,
 where $a$ and $b$ are some  functions on $\gamma$.
The unperturbed operator $H_0=-\Delta +W$ is self-adjoint in $L^2(\Real^2)$ and its spectrum is discrete.  Obviously, the operators $H_\eps$ are also  self-adjoint  with  discrete spectrum  and  $\dom H_\eps=\dom H_0$.
The main task is to describe the  limiting behaviour of the spectrum  of $H_\eps$ by constructing the asymptotics of eigenvalues $\lambda^\eps$ and eigenfunctions $u_\eps$ of the problem
\begin{equation}\label{SpectralEqn}
-\Delta u_\eps +(W+V_\eps) u_\eps= \lambda^\eps u_\eps\quad \hbox{in \ } \Real^2.
\end{equation}

\begin{figure}[b]
\centering
\includegraphics[scale=.7]{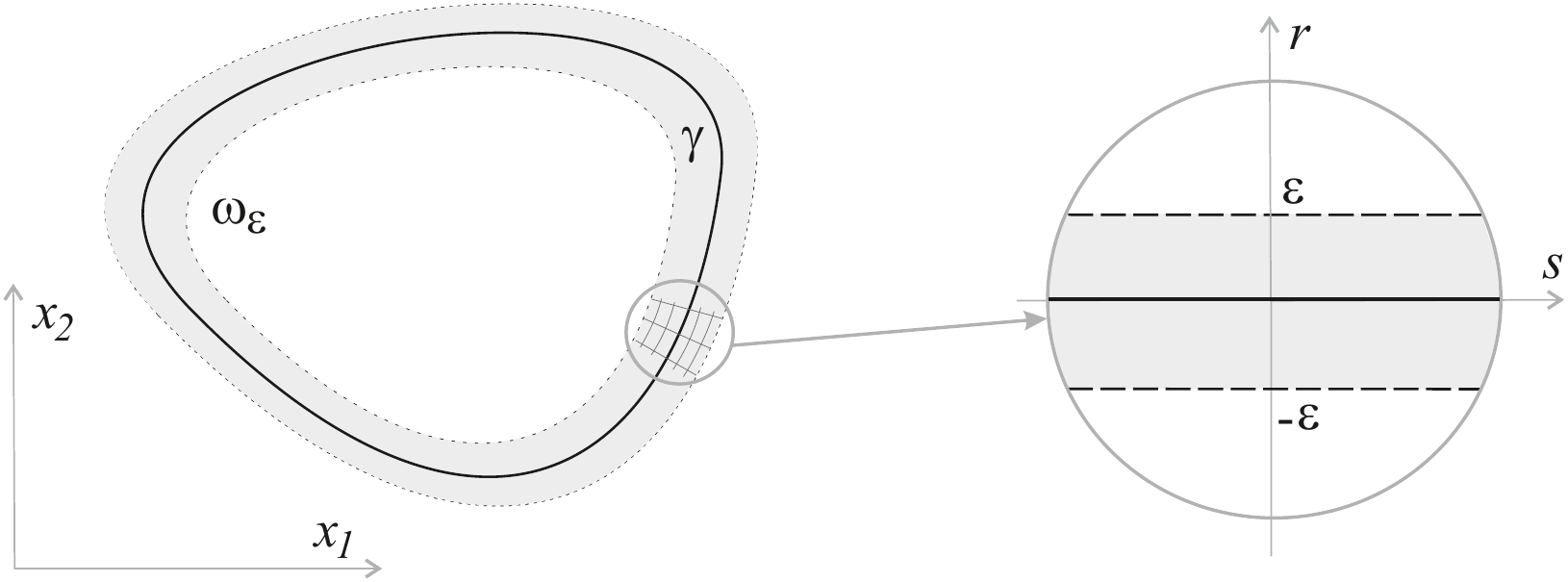}
\caption{The local coordinates  in $\omega_\eps$.}
\label{FigLC}
\end{figure}

We  say that the one-dimensional Schr\"odinger operator~$-\frac{d^2}{d r^2}+V$ in $L^2(\Real)$ possesses a \emph{zero-energy resonance}  if there exists a nontrivial solution~$h$ of the equation $-h'' +Vh= 0$ that is bounded on the whole line.  We call $h$ the \emph{half-bound state}. The half-bound state is  unique up to a scalar factor and has nonzero limits
\begin{equation*}
  h(-\infty)=\lim\limits_{r\to-\infty}h(r), \qquad
  h(+\infty)=\lim\limits_{r\to+\infty}h(r).
\end{equation*}
The closed curve $\gamma$ divides the plane into two domains $\Omega^-$ and $\Omega^+$, $\Real^2=\Omega^-\cup\gamma\cup\Omega^+$. Suppose  $\Omega^+$ is unbounded.
Let us introduce the space $\mathcal{W}^+\subset L^2(\Omega^+)$ as follows. We say that $f$ belongs to $\mathcal{W}^+$ if   there exists a functions $f_0$  belonging to  $\dom H_0$ such that $f=f_0$ in  $\Omega^+$. Also, we set $\mathcal{W}=\{f\in L^2(\Real^2)\colon f|_{\Omega^-}\in W_2^2(\Omega^-), \; f|_{\Omega^+}\in \mathcal{W}^+\}$.
Recall that $v^\pm$ denote the one-side traces of $v$ on $\gamma$.

Let $\cE\subset(0, 1)$ be an infinite set, for which zero is an accumulation point.
Our main result reads as follows.

\begin{thm}\label{MainThrmRes}
Assume that the operator~$-\frac{d^2}{d r^2}+V$ in $L^2(\Real)$ possesses a zero-energy resonance with the half-bound state $h$.

(i) Suppose that $\{\lme\}_{\eps\in\cE}$ is a sequence of eigenvalues of $H_\eps$ and $\{u_\eps\}_{\eps\in\cE}$ is the corresponding sequence of  eigenfunctions such that $\|u_\eps\|_{L^2(\Real^2)}=1$. If
\begin{equation}\label{LmneToLmn}
  \lme \to \lambda, \qquad u_\eps \to u\quad \text{in } L^2(\Real^2) \   \text{weakly}
\end{equation}
as $\cE\ni\eps\to 0$,  and $u$ is a non-zero function, then $\lambda$ is an eigenvalue with the eigenfunction $u$ of the operator $\cH=-\Delta +W$ in $L^2(\Real^2)$  acting on the domain
\begin{align*}
  \dom \cH=\big\{v\in \mathcal{W}\colon\;v^+=\theta v^-,\;\;\theta\partial_\nu v^+-\partial_\nu v^-
=\left(\textstyle\frac{1}{2 }(\theta^2-1)\kappa+\mu\right) v^-\,\text{on } \gamma \big\}.
\end{align*}
Here $\theta=h(+\infty)/h(-\infty)$, $\kappa=\kappa(s)$ is the  curvature of $\gamma$, and
\begin{equation*}
  \mu(s)=\frac{1}{h^2(-\infty)} \int_{\Real} U(s,r)h^2(r)\, dr.
\end{equation*}

(ii) If \eqref{LmneToLmn} holds and $\lambda$ is not a point of  $\sigma(\cH)$, then the sequence of eigenfunctions $u_\eps$ converges to zero as $\cE\ni\eps\to 0$ in the weak topology of  $L^2(\Real^2)$.

(iii) For each eigenvalue $\lambda$ of $\cH$ and  all $\eps$ small enough there exists an eigenvalue $\lme$ of  $H_\eps$ such that
$|\lme-\lambda|\leq c\eps$
with the constant $c$ depending only on $\lambda$.
\end{thm}

\begin{rem}\rm
The operator $-\frac{d^2}{d r^2}$ with the trivial potential $V=0$ possesses a zero-energy resonance with   the half-bound state $h=1$. Then $V_\eps(x) =\eps^{-1}\,U\left(s,\eps^{-1}r\right)$ and
$V_\eps\to \mu_0 \delta_\gamma$ in the space of distributions, where
\begin{equation}\label{Mu0}
  \mu_0(s)=\int_{\Real}U(s,r)\, dr.
\end{equation}
Since $\theta=1$, the interface conditions on $\gamma$
\begin{equation}\label{ConnectedCond}
  v^+=\theta v^-,\quad \theta\partial_\nu v^+-\partial_\nu v^-
=\left(\textstyle\frac{1}{2 }(\theta^2-1)\kappa+\mu\right) v^-
\end{equation}
become $v^+=v^-$, $\partial_\nu v^+-\partial_\nu v^-=\mu_0 v$, which correspond to the $\delta$-interaction supported on the curve \cite{BehrndtExnerHolzmannLotoreichik2017}.
\end{rem}

\begin{rem}\rm
 Note that the potentials $V$ for which the operator $-\frac{d^2}{d r^2}+V$ has a zero-energy resonance are not something exotic.
For any $V$ of compact support, there exists a discrete infinite set of real coupling constants $\alpha$ such that $-\frac{d^2}{d r^2}+\alpha V$  has a zero-energy resonance \cite{GolovatyManko2009, Golovaty2012}.
\end{rem}

Let us introduce two operators
\begin{align*}
  &\mathcal{D}^-= -\Delta+W\quad\text{in }L^2(\Omega^-), \qquad \dom \mathcal{D}^-=\{v\in W_2^2(\Omega^-)\colon\; v=0 \;\,\mbox{on } \gamma\},\\
  &\mathcal{D}^+= -\Delta+W\quad\text{in }L^2(\Omega^+), \qquad \dom \mathcal{D}^+=\{v\in \mathcal{W}^+\colon\; v=0 \;\,\mbox{on } \gamma\}.
\end{align*}

\begin{thm}\label{MainThrmNoRes}
Suppose that the operator $-\frac{d^2}{d r^2}+V$ in $L^2(\Real)$ has no zero-energy resonance and $\{\lme\}_{\eps\in\cE}$ is a sequence of eigenvalues of $H_\eps$ and $\{u_\eps\}_{\eps\in\cE}$ is the corresponding sequence of  eigenfunctions such that $\|u_\eps\|_{L^2(\Real^2)}=1$.

 (i) If $\lme \to \lambda$ and $u_\eps \to u$ in $L^2(\Real^2)$ weakly, as $\cE\ni\eps\to 0$, and the limit function $u$ is different from zero, then
$\lambda$ is an eigenvalue of the direct sum $\mathcal{D}^-\oplus\mathcal{D}^+$ and $u$ is the corresponding eigenfunction.

(ii) In the case when $\lme \to \lambda$, as $\cE\ni\eps\to 0$, and $\lambda\not\in \sigma(\mathcal{D}^-\oplus\mathcal{D}^+)$, the eigenfunctions $u_\eps$ converge to zero in $L^2(\Real^2)$ weakly.

(iii) If $\lambda\in \sigma(\mathcal{D}^-\oplus\mathcal{D}^+)$, then for all $\eps$ small enough we can find  an eigenvalue $\lme$ of $H_\eps$ such that
$ |\lme-\lambda|\leq c\eps$,
where the constant $c$ does not depend on $\eps$.
\end{thm}

The thin structure that is the support of  $V_\eps$ can produce a infinite series of  eigenvalues that go to the negative infinity as $\eps\to 0$.  Although for each $\eps>0$ the number of negative eigenvalues is finite, for some potentials $V_\eps$ it can increase infinitely as $\eps\to 0$. In particular, this means that the family of operators $H_\eps$ is not generally uniformly bounded from below with respect to  $\eps$.
In this case, any real number can be an accumulation point of the eigenvalues $\lme$ of $H_\eps$.
Theorems~\ref{MainThrmRes} and  \ref{MainThrmNoRes} point out the principal difference between the eigenvalues of the limit operators and all other real points. This difference is that only the points of $\sigma(\cH)$ in the resonant case or $\sigma(\mathcal{D}^-\oplus\mathcal{D}^+)$ in the non-resonant case can be approximated by the eigenvalues $\lme$ of $H_\eps$ so that the corresponding eigenfunctions $u_\eps$  converge to   nontrivial limits in $L^2(\Real^2)$.
The asymptotics of the negative low-lying eigenvalues remains an open problem for the time being.

%%%%%%%%%%%%%%%%%%%%%%%%%%%%%%%%%%%%%%%%%%%%%%%%%%%%%%%%%%%%%%%%%%%%%%%%%%
% Preliminaries
%%%%%%%%%%%%%%%%%%%%%%%%%%%%%%%%%%%%%%%%%%%%%%%%%%%%%%%%%%%%%%%%%%%%%%%%%%

\section{Preliminaries}
Returning  to the local coordinates $(s,r)$,
we see that the couple of vectors
$ \alpha=(\dot{\alpha}_1, \dot{\alpha}_2)$, $\nu=(-\dot{\alpha}_2, \dot{\alpha}_1)$
gives the Frenet frame for $\gamma$.
The Jacobian of transformation $x_1=\dot{\alpha}_1(s)-r\dot{\alpha}_2(s)$, $x_2=\dot{\alpha}_2(s)+r\dot{\alpha}_1(s)$ has the form
\begin{multline*}
J(s,r)=
\begin{vmatrix}
  \dot{\alpha}_1(s)-r\ddot{\alpha}_2(s)& -\dot{\alpha}_2(s)\\
          \dot{\alpha}_2(s)+r\ddot{\alpha}_1(s)\phantom{0} & \dot{\alpha}_1(s)
\end{vmatrix}
\\
=\dot{\alpha}_1^2(s)+\dot{\alpha}_2^2(s)
-r\big(\dot{\alpha}_1(s)\ddot{\alpha}_2(s)-
  \dot{\alpha}_2(s)\ddot{\alpha}_1(s)\big)=1-r \kappa(s),
\end{multline*}
where $\kappa=\det(\dot{\alpha},\ddot{\alpha})$ is the signed curvature of $\gamma$.
We see that $J$ is positive for sufficiently small $r$, because  the curvature $\kappa$  is  bounded on $\gamma$.
Namely, the  coordinates are defined correctly on  $\omega_\eps$ for all $\eps<\eps_*$, where
\begin{equation}\label{EpsStar}
	\eps_*=\min_{\gamma}|\kappa|^{-1}.
\end{equation}
Note also that $\kappa$ is defined uniquely up to the reparametrization $s\mapsto-s$.
Interface conditions \eqref{ConnectedCond} contain the parameters $\theta$, $\kappa$ and $\mu$ which depend on the particular parametrization chosen for curve $\gamma$. The parameters change along with the change of the Frenet frame.

\begin{prop}
  The operator $\mathcal{H}$ in Theorem~\ref{MainThrmRes} does not depend upon the choice of the Frenet frame for  $\gamma$.
\end{prop}
\begin{proof}
Every smooth  curve  admits two possible orientations of the arc-length parameter and consequently two possible  Frenet frames. Let us change the Frenet frame $\{\alpha, \nu\}$ to the frame $\{-\alpha, -\nu\}$ and prove that  conditions \eqref{ConnectedCond} will remain the same. The change leads to the following transformations:
\begin{equation*}
h(\pm\infty)\mapsto h(\mp\infty),\;\;
\theta\mapsto \theta^{-1}, \;\; \kappa\mapsto -\kappa,\;\; \mu\mapsto \theta^{-2}\mu, \;\; u_\pm\mapsto u_\mp, \;\; \partial_\nu u_\pm \mapsto -\partial_\nu u_\mp.
\end{equation*}
The first condition $u^+-\theta u^-=0$ in \eqref{ConnectedCond} transforms into $u^--\theta^{-1} u^+=0$ and therefore remains unchanged. For the second condition, we obtain
\begin{equation*}
 	-\theta^{-1}\partial_\nu u^-+\partial_\nu u^+
	-\left(-\textstyle\frac{1}{2}(\theta^{-2}-1)\kappa
	+\theta^{-2}\mu\right) u^+=0.
\end{equation*}
Multiplying the equality by $\theta$  yields
\begin{equation*}
	\theta\partial_\nu u^+-\partial_\nu u^-
	-\left(\textstyle\frac{1}{2 }(\theta^{2}-1)\kappa+\mu\right) 	\theta^{-1} u^+=0,
\end{equation*}
since $-\theta(\theta^{-2}-1)=\theta^{-1}(\theta^{2}-1)$. It remains to insert $u^-$ in place of $\theta^{-1} u^+$, in view of the first interface condition.
\end{proof}

The metric tensor $g=(g_{ij})$ in the orthogonal coordinates $(s,r)$  has the form
  \begin{equation*}
    g=
    \begin{pmatrix}
      J^2\phantom{0} & 0 \\
          0\phantom{0} & 1
    \end{pmatrix}.
  \end{equation*}
In fact, we have
$g_{11}=|x_s|^2=|\dot{\alpha}+r \dot{\nu}|^2
=|(1-r\kappa) \dot{\alpha}|^2=J^2$,
by the Frenet-Serret formula $\dot{\nu}=-\kappa \dot{\alpha}$, and $g_{22}=|x_r|^2=|\nu|^2=1$.
Then the gradient and the  Laplace-Beltrami operator in the local coordinates become
\begin{equation}\label{GradientLaplasianInSN}
\nabla \phi=J^{-1}\partial_s\phi\, \alpha+\partial_r\phi\, \nu,\qquad
\Delta \phi=J^{-1}\left(\partial_s(J^{-1}\partial_s \phi)+ \partial_r(J\partial_r \phi)\right).
\end{equation}

All the results presented in Theorems~\ref{MainThrmRes} and~\ref{MainThrmNoRes}  concern arbitrary potentials $V_\eps$ of the form \eqref{Veps} that generally diverge in the distributional sense. However,  the spectra of $H_\eps$ converge to the spectra of the limit operators  without reference to the convergence of  potentials. The following statement shows that the convergence conditions for $V_\eps$ and for the spectra of $H_\eps$ are quite different.

\begin{prop}
The family of potentials $V_\eps$ converges in the space of distributions if and only if
$\int_\Real V\,dr=0$. In this case,
\begin{equation*}
  V_\eps\to \mu_1\,\partial_\nu\delta_\gamma+\left(\mu_1\kappa+\mu_0\right) \delta_\gamma \quad \text{in }\mathcal{D}'(\Real^2),
\end{equation*}
where $\mu_1=-\int_\Real r V(r)\,dr$ and $\mu_0$ is given by \eqref{Mu0}.
\end{prop}
\begin{proof}
It is evident that the sequence $\eps^{-1}\,U\left(s,\eps^{-1}r\right)$
converges to  $\mu_0 \delta_\gamma$ in $\mathcal{D}'(\Real^2)$.
Write $g_\eps=\eps^{-2}\,V\left(\eps^{-1}r\right)$  and $n=\eps^{-1}r$. Then we have
\begin{multline*}
\int_{\Real^2}g_\eps\phi\,dx
=\int_{\omega_\eps}g_\eps\phi\,dx
=
\eps^{-2}\int_{-\eps}^\eps\int_{0}^{|\gamma|} V(\eps^{-1}r)\phi(s,r)(1-r\kappa(s))\,ds\,dr
\\
 =
\eps^{-1}\int_{-1}^1\int_0^{|\gamma|} V(n)\phi(s,\eps n)(1-\eps n\kappa(s))\,ds\,dn
=\eps^{-1}\int_{-1}^1 V(n)\,dn \int_0^{|\gamma|}\phi(s,0)\,ds
\\
+
\int_{-1}^1 n V(n)\,dn \int_0^{|\gamma|}\big(\partial_n\phi(s,0)-\kappa(s)\phi(s,0)\big)\,ds+O(\eps)
\end{multline*}
as $\eps\to 0$  for all $\phi\in C^\infty_0(\Real^2)$.
The  sequence $g_\eps$ has a finite limit in $\mathcal{D}'(\Real^2)$ iff $\int_\Real V\,dn=0$.
In this case, we have
\begin{equation*}
\int_{\Real^2}g_\eps\phi\,dx\to \mu_1\int_\gamma\left(\partial_\nu\delta_\gamma+\kappa \delta_\gamma\right)\phi\,d\gamma,
\end{equation*}
which completes the proof.
\end{proof}

%%%%%%%%%%%%%%%%%%%%%%%%%%%%%%%%%%%%%%%%%%%%%%%%%%%%%%%%%%%%%%%%%%%%%%%%%%
% Formal Asymptotics
%%%%%%%%%%%%%%%%%%%%%%%%%%%%%%%%%%%%%%%%%%%%%%%%%%%%%%%%%%%%%%%%%%%%%%%%%%

\section{Formal Asymptotics}

Now we will show how interface conditions \eqref{ConnectedCond} can be found by direct calculations.
Here we use the asymptotic methods similar to those in  \cite{GolLavr2000, GolGomLoboPer2004, GomNazarovPer2020}.
In the sequel, the normal vector field $\nu$ on $\gamma$ will be outward to the domain $\Omega^-$. Hence the local coordinate $r$ increases in the direction from $\Omega^-$ to $\Omega^+$. Also, it will be convenient to parameterize the curve $\gamma$ by  points of a circle. It  will allow us not to indicate every time that  functions on $\gamma$  are periodic on $s$.
Let $S$ be the circle of the  length $|\gamma|$.
Then $\omega_\eps$ is diffeomorphic to the cylinder $Q_\eps=S\times (-\eps, \eps)$.
We  denote by $\gamma_t$ the  curve that is obtained from $\gamma$ by flowing for ``time'' $t$ along the normal vector field, i.e.,
  $\gamma_t=\{x\in\Real^2\colon\; x=\alpha(s)+t\nu(s), \; s\in S\}$.
Then the boundary of $\omega_\eps$ consists of two curves $\gamma_{-\eps}$ and $\gamma_{\eps}$.

We  look for the  approximation to the eigenvalue $\lambda_\eps$ and the corresponding eigenfunction $u_\eps$ of \eqref{SpectralEqn} in the form
\begin{equation}\label{AsymptoticsUe}
\lambda^\eps\approx \lambda, \qquad u_\eps(x)\approx
\begin{cases}
  u(x)& \hbox{in \ }\Real^2\setminus \omega_\eps, \\
    v_0\left(s,\nep\right)+\eps v_1\left(s,\nep\right)+\eps^2 v_2\left(s,\nep\right)
&\hbox{in \ } \omega_\eps.
\end{cases}
\end{equation}
To match the approximations in $\omega_\eps$ and $\Real^2\setminus \omega_\eps$, we hereafter assume that
\begin{equation}\label{MatchingCnds}
  [u_\eps]_{\pm\eps}=0, \quad [\partial_r u_\eps]_{\pm\eps}=0,
\end{equation}
where  $[w]_t$ stands for  the jump of $w$ across $\gamma_t$ in the positive direction of the local coordinate $r$.
Since  $u_\eps$ solves \eqref{SpectralEqn} and the domain $\omega_\eps$ shrinks to $\gamma$, the function $u$ must be a solution of the equation
\begin{equation}\label{EqnForU}
-\Delta u+Wu= \lambda u \quad \hbox{in \ } \Real^2\setminus \gamma
\end{equation}
subject to appropriate transmission conditions on $\gamma$.
To find these conditions, we consider equation \eqref{SpectralEqn} in the local coordinates $(s,n)$, where $n=r/\eps$. By \eqref{GradientLaplasianInSN},
the Laplacian can be written as
\begin{equation*}
  \Delta =\frac1{1-\eps n\kappa}\left( \eps^{-2}\partial_n
  (1-\eps n\kappa)\partial_n +\partial_s
  \Big(\frac1{1-\eps n\kappa}\,\partial_s\Big)\right),
\end{equation*}
in the cylinder $Q=S\times (-1,1)$.
From this we readily deduce the  representation
\begin{equation}\label{LaplaceExpansion}
\Delta= \eps^{-2}\partial^2_n-\eps^{-1}\kappa(s)\partial_n
-n\kappa^2(s)\partial_n+\partial^2_s+\eps P_\eps,
\end{equation}
where $P_\eps$ is a PDE of the second order on $s$ and the first one on $n$ whose coefficients  are uniformly bounded in $Q$ with respect to $\eps$.

Substituting $v_0+\eps v_1+\eps^2 v_2$ and \eqref{LaplaceExpansion} into  equation \eqref{SpectralEqn} in particular yields
\begin{gather}\nonumber
-\pte^2 v_0+Vv_0=0,
\quad
-\pte^2 v_1+Vv_1=-\kappa\pte v_0-Uv_0,\\\label{EqnV2}
 -\pte^2 v_2+Vv_2=-(\kappa\pte +U) v_1
  +(\partial^2_s-n\kappa^2\partial_n-W(\,\cdot\,,0)+\lambda)v_0
\end{gather}
in  $Q$.
From \eqref{MatchingCnds} we see that necessarily
\begin{gather}\label{FittingCndsUV0}
 u^-(s)=v_0(s,-1),\qquad u^+(s)=v_0(s,1),
 \\\label{FittingCndspV0}
 \partial_n v_0(s,- 1)=0, \qquad \partial_n v_0(s, 1)=0, \\\label{FittingCndspV1}
 \partial_n v_1(s, -1)=\partial_r u^-(s), \qquad
 \partial_n v_1(s, 1)=\partial_r u^+(s).
\end{gather}
Combining \eqref{FittingCndspV0}--\eqref{FittingCndspV1}, we conclude that $v_0$ and $v_1$ solve the  problems
\begin{align}\label{problemV0}
&\begin{cases}
  -\pte^2 v_0+V(n)v_0=0 \quad \hbox{in \ } Q, \\
    \phantom{-}\partial_n v_0(s,- 1)=0, \quad \partial_n v_0(s, 1)=0, \quad s\in S;
\end{cases}
\\\label{problemV1}
&\begin{cases}
  -\pte^2 v_1+V(n)v_1=-\kappa(s)\pte v_0-U(s,n)v_0\quad \hbox{in \ } Q, \\
    \phantom{-}\partial_n v_1(s, -1)=\partial_r u^-(s), \quad
\partial_n v_1(s, 1)=\partial_r u^+(s), \quad s\in S
\end{cases}
\end{align}
respectively. Hence we have the boundary value problems in $Q$ including the ``non-ellip\-tic'' partial differential operator $-\pte^2+V$. These problems can also be  regarded as the boun\-da\-ry value problems on $\cI$ for ordinary differential equations which depend on the parameter $s\in S$.

\subsection{Case of zero-energy resonance}
Assume that  $-\frac{d^2}{dr^2}+V$ has a zero energy resonance with  the half-bound state $h$. Set $\cI=(-1,1)$. Since the support of $V$ lies in $\overline{\cI}$, the  function $h$ is  constant  outside $\cI$ as a bounded solution of the equation $h''=0$.
Therefore the restriction of $h$ to $\cI$ is a nonzero solution of the Neumann boundary value problem
\begin{equation}\label{NeumanProblem}
     -h''+Vh=0 \text{ \ in }\cI,\qquad   h'(-1)=0, \quad h'(1)=0.
\end{equation}
Hereafter, we fix $h$ by the additional condition $h(-1)=1$. Then $h(\pm\infty)=h(\pm 1)$ and $\theta=h(1)$.

In this case, \eqref{problemV0}  admits the infinitely many solutions $v_0(s,n)=a_0(s)h(n)$, where $a_0$ is an arbitrary function on $S$.
From \eqref{FittingCndsUV0} we deduce that
\begin{equation*}
  u^-=a_0, \qquad u^+=h(1) a_0=\theta a_0
\end{equation*}
and hence that $v_0(s,n)=u^-(s)h(n)$ and
\begin{equation}\label{RCond0}
     u^+=\theta u^-\quad\text {on }\gamma.
\end{equation}

Next, problem \eqref{problemV1} is in general unsolvable, since \eqref{problemV0} admits  nontrivial solutions.  To find solvabi\-li\-ty conditions, we rewrite the equation in \eqref{problemV1} in the form
$ -\pte^2 v_1+V(n)v_1=-\big(\kappa(s)h'(n)+U(s,n)h(n)\big)u^-(s)$,
multiply  by $a(s)h(n)$, $a\in L^2(S)$,  and then integrate over $Q$
\begin{multline}\label{IntV1H}
\int_{Q}\left(-\pte^2 v_1+V(n)v_1\right)a(s)h(n)\,dn\,ds
\\
=
-\int_{Q}\big(\kappa(s)h'(n)+U(s,n)h(n)\big)u^-(s)a(s)h(n)\, dn\,ds.
\end{multline}
Since $h$ is a solution of \eqref{NeumanProblem}, integrating by parts twice on the left-hand side yields
\begin{multline*}
\int_{S} \int_{\cI}\left(-\pte^2 v_1+Vv_1\right)a h\,dn \,ds
=-\int_{S}( \partial_n v_1 h-v_1 h')\big|_{n=-1}^{n=1}a\,ds\\-
\int_{S} \int_{\cI} a v_1\left(-h''+Vh\right)\,dn\,ds
=-\int_{S}\big(\theta\partial_r u^+-\partial_r u^-\big) a\,ds,
\end{multline*}
in view of the boundary conditions for $v_1$.
Hence \eqref{IntV1H} becomes
\begin{equation*}
\int_{S}\left(\theta\partial_r u^+-\partial_r u^-\right)a\,ds
=\int_{S} u^-a\int_{\cI}\left(\kappa hh'+U h^2\right)\,dn\,ds.
\end{equation*}
The equality  $hh'=\frac12 (h^2)'$ implies
\begin{equation}\label{IntHHpr}
\int_{\cI}hh'\,dn=\tfrac12 (h^2(1)-h^2(-1))=\tfrac{1}{2 }(\theta^2-1).
\end{equation}
 Therefore we obtain
\begin{equation*}
\int_{S}\left(\theta\partial_r u^+-\partial_r u^-\right)a\,ds
=\int_{S}\big(\textstyle\frac{1}{2}(\theta^2-1)\kappa+\mu \big)u^-a\,ds
\end{equation*}
for all  $a\in L^2(S)$, where $\mu(s)=\int_{\cI} U(s,n)h^2(n)\, dn$.
From this we deduce
\begin{equation*}
  \theta\partial_r u^+-\partial_r u^-
=\big(\textstyle\frac{1}{2}(\theta^2-1)\kappa+\mu \big) u^-\quad\text {on }\gamma,
\end{equation*}
which is necessary for  solvability of \eqref{problemV1}.
In view of the Fredholm alternative, this condition is also  sufficient. Moreover it is a jump condition for the normal derivative of $u$   at the interface $\gamma$, since $\partial_\nu u^\pm=\partial_r u^\pm$ on $\gamma$.
Therefore  $\lambda$ and $u$ in  \eqref{AsymptoticsUe} must solve the problem
\begin{gather}\label{LimitProblemEq}
-\Delta u+Wu=\lambda u \qquad \hbox{in  } \Real^2\setminus \gamma,
\\ \label{RConds}
 \phantom{-}u^+-\theta u^-=0,  \quad
\theta\partial_\nu u^+-\partial_\nu u^-
=\big(\textstyle\frac{1}{2}(\theta^2-1)\kappa+\mu \big) u^- \quad \hbox{on } \gamma,
\end{gather}
 which can be equivalently rewritten as the spectral equation $\cH u=\lambda u$.

Assume that $\lambda$ is an eigenvalue of  $\cH$ and $u$  is an eigenfunction for this eigenvalue.
Now we can calculated the trace $u^-$ on $\gamma$ and  finally determine $v_0(s,n)=u^-(s)h(n)$.
Since the second condition in \eqref{RConds} holds,
problem \eqref{problemV1} is solvable and $v_1$ is defined up to the term $a_1(s) h(n)$.
Let us fix  a solution of \eqref{problemV1} so that
\begin{equation}\label{V1At-1}
  v_1(s,-1)=0, \quad s\in S.
\end{equation}
Finally equation \eqref{EqnV2} admits a unique solution $v_2$ satisfying  the conditions
\begin{equation}\label{V2At-1}
  v_2(s, -1)=0, \quad \partial_n v_2(s, -1)=0, \quad s\in S.
\end{equation}
The functions $v_k$ are smooth in $Q$ due to the smoothness of  $V$, $U$ and $\kappa$. Recall also that $W$ is smooth in an neighbourhood of $\gamma$.
So we have constructed all terms in asymptotics \eqref{AsymptoticsUe}.

\subsection{Non-resonant case}
Now we suppose that  $-\frac{d^2}{dr^2}+V$ has no zero energy resonance. Then  problem \eqref{NeumanProblem}, and hence problem \eqref{problemV0}, admit the trivial solutions $h=0$ and $v_0=0$ only,  and    \eqref{FittingCndsUV0} imply $u^-=0$ and $u^+=0$ on $\gamma$. We thus get
\begin{equation*}
-\Delta u+Wu=\lambda u\quad \hbox{in \ } \Real^2\setminus \gamma,\qquad
 u|_{\gamma}=0.
\end{equation*}
Let us suppose that $\lambda$ is an eigenvalue of the direct sum
$\mathcal{D}^-\oplus\mathcal{D}^+$ and $u$ is the corresponding eigenfunction.
In this case, problem \eqref{problemV1} has the form
\begin{equation*}
\begin{cases}
    -\pte^2 v_1+V(n)v_1=0\quad \hbox{in \ } Q, \\
    \phantom{-}\partial_n v_1(s, -1)=\partial_\nu u^-, \qquad
\partial_n v_1(s, 1)=\partial_\nu u^+
\end{cases}
\end{equation*}
and admits a unique solution. Let us substitute $v_0=0$ into equation \eqref{EqnV2} and assume that $v_2$ is a solution of the Cauchy problem
\begin{equation*}
\begin{cases}
-\pte^2 v_2+V(n)v_2=-\kappa(s)\pte v_1 +U(s,n) v_1
  \quad \hbox{in \ } Q,
\\
   \phantom{-} v_2(s, -1)=0,
 \quad
 \partial_n v_2(s, -1)=0, \quad s\in S.
\end{cases}
\end{equation*}

\subsection{Quasimodes of $H_\eps$}
To prove that $\lambda$ belonging to either $\sigma(\cH)$ or $\sigma(\mathcal{D}^-\oplus\mathcal{D}^+)$ is an accumulation point for a sequence of eigenvalues $\lambda^\eps$ of $H_\eps$, we will apply the method of quasimodes.
Let $A$ be a self-adjoint operator in a Hilbert space $L$.
We say a pair $(a, \phi)\in \Real\times \dom A$ is a \textit{quasimode} of  $A$ with the accuracy $\delta$, if $\phi\neq 0$  and
$\|(A-a)\phi\|_L\leq\delta\|\phi\|_L$.

\begin{prop}[\hglue-0.1pt{\cite[p.139]{PDEVinitiSpringer}}]\label{LemQuasimodes}
   Assume $(a, \phi)$ is a quasimode of $A$ with accuracy $\delta>0$ and  the spectrum of $A$ is discrete in  the interval
$[a-\delta, a+\delta]$. Then there exists an eigenvalue $\lambda_*$ of  $A$ such that $|\lambda_*-a|\leq\delta$.
\end{prop}

Since  its  proof  is  so  simple,  we  reproduce  it  here  for  the  reader's
convenience.
If $a\in \sigma(A)$, then $\lambda_*=a$. Otherwise  the distance $d_a$ from $a$ to the spectrum of $A$  can be computed as
\begin{equation*}
  d_a=\|(A-a)^{-1}\|^{-1}
  =\inf_{\psi\neq0}\frac{\|\psi\|_L}{\|(A-a)^{-1}\psi\|_L},
\end{equation*}
where $\psi$ is an arbitrary vector of $L$. Taking $\psi=(A-a)\phi$, we deduce
\begin{equation*}
  d_a\leq \frac{\|(A-a)\phi\|_L}{\|\phi\|_L}\leq \delta,
\end{equation*}
from which the assertion  follows.

\begin{figure}[b]
  \centering
  \includegraphics[scale=1.8]{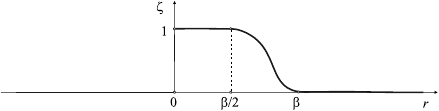}\\
  \caption{Plot of the function $\zeta$.}\label{FigPlotZeta}
\end{figure}
\smallskip

In order to construct the quasimodes of $H_\eps$, we must modify the approximation
\begin{equation*}
 \hat{v}_\eps(x)=
\begin{cases}
  u(x)& \hbox{in \ }\Real^2\setminus \omega_\eps, \\
    v_0\left(s,\nep\right)+\eps v_1\left(s,\nep\right)+\eps^2 v_2\left(s,\nep\right)
&\hbox{in \ } \omega_\eps
\end{cases}
\end{equation*}
obtained above.   The approximation does not in general belong to $\dom H_\eps$, because $\hat{v}_\eps$ has jump discontinuities on  $\partial\omega_\eps$.
Let us define the function $\zeta$ plotted in Fig.~\ref{FigPlotZeta}. This function is smooth outside the origin,  $\zeta(r)=1$ for $r\in [0,\beta/2]$ and $\zeta(r)=0$ in the set $\Real\setminus [0,\beta)$.
We  assume that $2\beta<\eps_*$, where $\eps_*$ is given by \eqref{EpsStar}. Set
\begin{equation}\label{EtaEps}
  \eta_\eps=\big([\hat{v}_\eps]_{\eps}+[\partial_\nu\hat{v}_\eps]_{\eps}\,
  (r-\eps)\big)\,\zeta(r-\eps)
  +\big([\hat{v}_\eps]_{-\eps}+[\partial_\nu\hat{v}_\eps]_{-\eps}\,(r+\eps)\big)
  \,\zeta(-r-\eps).
\end{equation}
It is easy to check  that $\eta_\eps$ and $\partial_r\eta_\eps$ have the same jumps across the boundary of $\omega_\eps$ as $\hat{v}_\eps$ and $\partial_\nu\hat{v}_\eps$ respectively. In addition, $\eta_\eps$ is different from zero in the set $\omega_{\beta+\eps}\setminus\omega_\eps$ only. Therefore the function
\begin{equation*}
 v_\eps(x)=
\begin{cases}
  u(x)-\eta_\eps(x)& \hbox{in \ }\Real^2\setminus \omega_\eps, \\
    v_0\left(s,\nep\right)+\eps v_1\left(s,\nep\right)+\eps^2 v_2\left(s,\nep\right)
&\hbox{in \ } \omega_\eps
\end{cases}
\end{equation*}
belongs to the domain of $\cH_\eps$. We have not changed $\hat{v}_\eps$ too much, since
\begin{equation}\label{EtaEpsEstimate}
  \sup_{x\in \Real^2\setminus \overline{\omega}_\eps}\big(|\eta_\eps(x)|+|\Delta\eta_\eps(x)|\big)\leq c\eps.
\end{equation}
It follows from  explicit formula \eqref{EtaEps} and the smallness of
 jumps of $\hat{v}_\eps$ and $\partial_\nu\hat{v}_\eps$ across $\partial\omega_\eps$. Indeed,
using \eqref{FittingCndsUV0}--\eqref{FittingCndspV1}, \eqref{V1At-1} and \eqref{V2At-1} for the case of resonance we deduce
 \begin{align*}{}
  &[\hat{v}_\eps]_{-\eps}=v_0(s,-1)-u(s,-\eps)
                  =u^-(s)-u(s,-\eps)=O(\eps),
  \\
  &\begin{aligned}{}
  [\hat{v}_\eps]_{\eps}=u(s,\eps)&-v_0(s,1)-\eps v_1(s,1)-\eps^2 v_2(s,1)\\
                  &=u(s,\eps)-\theta u^-(s)+O(\eps)
=u(s,\eps)-u^+(s)+O(\eps)=O(\eps),
   \end{aligned}
   \\
   &\begin{aligned}{}
  [\partial_\nu\hat{v}_\eps]_{-\eps}=\eps^{-1}\partial_n v_0(s,-1)+\partial_n v_1(s,&-1)-\partial_r u(s,-\eps)
                  \\
  &=\partial_r u^-(s)-\partial_r u(s,-\eps)=O(\eps),
   \end{aligned}
   \\
    &\begin{aligned}{}
  [\partial_\nu\hat{v}_\eps]_{\eps}=\partial_r u(s,\eps)-\eps^{-1}\partial_n v_0(s,1)&-\partial_nv_1(s,1)-\eps\partial_n v_2(s,1)+O(\eps)
                  \\
  &=\partial_r u^+(s)-\partial_r u(s,\eps)+O(\eps)=O(\eps),
   \end{aligned}
 \end{align*}
as $\eps\to 0$. Here we also have utilized  condition \eqref{RCond0} and the  inequality
\begin{equation*}
|u(s,\pm \eps)-u^\pm(s)|+|\partial_ru(s,\pm \eps)
-\partial_r u^\pm(s)|\leq c\eps.
\end{equation*}
Note that the eigenfunction $u$ is smooth in a neighbourhood of $\gamma$.
Obviously the jumps are also of order $O(\eps)$ in the non-resonant case,
when $v_0=0$ and $u^\pm=0$.

\begin{lem}\label{LemmaQuasimodesHeps}
 The pairs $(\lambda, v_\eps)$ constructed above are quasimodes of $H_\eps$ with the accuracy $O(\eps)$ as $\eps\to 0$.
\end{lem}
\begin{proof}

Write $\varrho_\eps=(H_\eps-\lambda)v_\eps$.
Thus \eqref{EqnForU} implies
\begin{equation*}
 \varrho_\eps=(-\Delta+W-\lambda)( u-\eta_\eps)=(-\Delta+W-\lambda)\eta_\eps
\end{equation*}
outside $\omega_\eps$. Therefore $\sup_{x\in\Real^2\setminus \omega_\eps}|\varrho_\eps(x)|\leq c_1\eps$, because of \eqref{EtaEpsEstimate}. Recall $\eta_\eps$ is a function of compact support.
Applying representation \eqref{LaplaceExpansion} of the Laplace operator in the local coordinates, we deduce
\begin{multline*}
-\Delta+W(x)+V_\eps(x)
=-\eps^{-2}\partial^2_n+\eps^{-1}\kappa\partial_n
+n\kappa^2\partial_n-\partial^2_s-\eps P_\eps
+W(s,\eps n)\\
+ \eps^{-2}V(n)+\eps^{-1}U(s,n) =\eps^{-2}\ell_0+\eps^{-1}\ell_1+\ell_2+W(s,\eps n)-\eps P_\eps
\end{multline*}
for $x\in \omega_\eps$, where $\ell_0=-\partial^2_n+V$, $\ell_1=\kappa\partial_n
+U$ and $\ell_2=n\kappa^2\partial_n-\partial^2_s$. Then
\begin{multline*}
  \varrho_\eps=(-\Delta+W+V_\eps -\lambda)v_\eps
  =\big(\eps^{-2}\ell_0+\eps^{-1}\ell_1+\ell_2+W(s,\eps n)-\eps P_\eps-\lambda\big)\big(v_0+\eps v_1+\eps^2 v_2\big)\\
  =\eps^{-2}\ell_0v_0+\eps^{-1}(\ell_0v_1+\ell_1v_0)+\big(\ell_0v_2+\ell_1v_1
  +(\ell_2+W(s,0)-\lambda)v_0\big)
  \\
  +(W(s,\eps n)-W(s,0))v_0+\eps \big(\ell_1v_2+(\ell_2+W(s,\eps n)-\lambda)
  (v_1+\eps v_2) - P_\eps v_\eps\big)
\end{multline*}
for $x\in \omega_\eps$. From our choice of $v_k$, we derive that the first three terms of the right-hand side vanish. The potential $W$ is a $C^\infty$-function  in a neighbourhood of $\gamma$, then we have
$\sup_{x\in\omega_\eps}|\varrho_\eps(x)|\leq c_2\eps$.
Hence
\begin{equation*}
  \|(H_\eps-\lambda)v_\eps\|_{L^2(\Real^2)}=
  \|\varrho_\eps\|_{L^2(\Real^2)}\leq |\omega_{2\beta}|^{1/2}\sup_{\Real^2}|\varrho_\eps|\leq c_3 \eps,
\end{equation*}
since $\supp \varrho_\eps\subset \omega_{\beta+\eps}\subset \omega_{2\beta}$ for $\eps$ small enough.
On the other hand, the main contribution  to the $L^2(\Real^2)$-norm of $v_\eps$ is given by the eigenfunction $u$. Therefore $\|v_\eps\|_{L^2(\Real^2)}\geq \frac{1}{2}\|u\|_{L^2(\Real^2)}$ for $\eps$ small enough.
 Finally, we obtain
\begin{equation*}
 \|(H_\eps-\lambda)v_\eps\|_{L^2(\Real^2)}\leq c_3 \eps
 \leq 2c_3 \eps\|u\|_{L^2(\Real^2)}^{-1}\,\|v_\eps\|_{L^2(\Real^2)}\leq c_4 \eps\,\|v_\eps\|_{L^2(\Real^2)},
\end{equation*}
and this is precisely the assertion of the lemma.
\end{proof}

%%%%%%%%%%%%%%%%%%%%%%%%%%%%%%%%%%%%%%%%%%%%%%%%%%%%%%%%%%%%%%
% Proof of Main Results
%%%%%%%%%%%%%%%%%%%%%%%%%%%%%%%%%%%%%%%%%%%%%%%%%%%%%%%%%%%%%%

\section{Proof of Main Results}
Let $\{\lambda^\eps\}_{\eps\in\cE}$ be a sequence of eigenvalues of  $H_\eps$ and $\{u_\eps\}_{\eps\in\cE}$ be the sequence of the corresponding eigenfunctions and $\|u_\eps\|_{L^2(\Real^2)}=1$. Let $\chi_{\eps}$ be the characteristic function of $\Real^2\setminus\omega_\eps$.

\begin{lem}\label{LemmaDconv}
  Assume that $\lambda^\eps\to \lambda$ and $u_\eps\to u$ in $L^2(\Real^2)$ weakly as $\cE\ni\eps\to 0$.

  (i) For any bounded or unbounded domain $D$ in $\Real^2$ such that $\overline{D}\cap\gamma=\emptyset$ the  eigenfunctions $u_\eps$ converge to $u$ in $W_2^2(D)$ weakly, and  $u$ solves the equation
  \begin{equation}\label{EqnForUoutside}
  -\Delta u+Wu=\lambda u\quad \text{in } \Real^2\setminus \gamma.
  \end{equation}

  (ii) $\chi_{\eps}\nabla u_\eps\to\nabla u$ in $L^2(\Omega^\pm)$ weakly.

 (iii) Treating $u_\eps(x)$ as $u_\eps(s,r)$, we have
 \begin{equation*}
   u_\eps(\,\cdot\,,-\eps)\to u^-\;\;\text{and}\;\; u_\eps(\,\cdot\,,\eps)\to u^+\;\;\text{in }L^2(S)\;\; \text{weakly}.
 \end{equation*}
\end{lem}

\begin{proof}
\textit{(i)} Recall that  $\supp V_\eps$ lies in $\omega_\eps$ and chose $\eps$  so small that $D\cap \omega_\eps=\emptyset$. Then for any $\phi\in C_0^\infty(D)$ we conclude from \eqref{SpectralEqn} that
 \begin{equation*}
    \int_{D} \Delta u_\eps\phi\,dx=\int_{D} (W-\lambda^\eps)\, u_\eps\phi\,dx.
  \end{equation*}
The right-hand side  has a limit as $\cE\ni\eps\to 0$ by the assumptions, thus the left-hand side also converges for all $\phi\in C_0^\infty(D)$, i.e., $\Delta u_\eps\to \Delta u$ in $L^2(D)$ weakly. From this we deduce that $u_\eps$ converges to $u$ in  $W_2^2(D)$ weakly, and hence that
 \begin{equation*}
    \int_{D} \Delta u\phi\,dx=\int_{D} (W-\lambda)\,u\phi\,dx.
  \end{equation*}
Since $D$ is an arbitrary domain such that $\overline{D}\cap\gamma=\emptyset$,
we have
  \begin{equation*}
    \int_{\Real^2} \Delta u\phi\,dx=\int_{\Real^2} (W-\lambda)\,u\phi\,dx
  \end{equation*}
for all test functions $\phi\in C^\infty_0(\Real^2)$ for which $\supp \phi\cap\gamma=\emptyset$.
Therefore $u$ is a solution of \eqref{EqnForUoutside}.

\textit{(ii)}
We conclude from
 \begin{equation*}
 \int_{\Real^2\setminus\omega_\eps} \Delta u_\eps\psi\,dx=\int_{\Real^2\setminus\omega_\eps} (W-\lambda^\eps)\,u_\eps\psi\,dx,\quad \psi\in C^\infty_0(\Real^2)
  \end{equation*}
that the family of functionals $\chi_\eps \Delta u_\eps$ in $L^2(\Real^2)$
is pointwise bounded, since the right-hand side is bounded as $\eps\to 0$. In view of  the uniform boundedness principle, we have $\|\chi_\eps \Delta u_\eps\|_{L^2(\Real^2)}\leq c_1$,
from which the estimate $ \|u_\eps\|_{W_2^2(\Real^2\setminus\omega_\eps)}\leq c_2$ follows. Now for $\psi \in C_0^\infty(\Omega^+)$ and $\eps$ small enough, we have
\begin{equation*}
\int_{\Omega^+}(\chi_{\eps}\nabla u_\eps-\nabla u)\psi\,dx
  =\int_{\supp \psi}(\nabla u_\eps-\nabla u)\psi\,dx\to 0
\end{equation*}
as $\eps\to 0$, in view of \textit{(i)}.
Therefore $\chi_{\eps}\nabla u_\eps\to\nabla u$ in $L^2(\Omega^+)$ weakly, because
%the sequence $\chi_{\eps}\nabla u_\eps$ is uniformly bounded on $\eps$ and
$C_0^\infty(\Omega^+)$ is dense in $L^2(\Omega^+)$. Similar considerations apply to $\Omega^-$.

\textit{(iii)}
Choose the  cutoff function
\begin{equation*}
  \zeta_\eps(r)=
  \begin{cases}
    (r-\eps)\zeta(r) & \text{if } r\geq \eps,\\
    \phantom{mm}0 &\text{otherwise,}
  \end{cases}
\end{equation*}
where $\zeta$ is  plotted in Fig.~\ref{FigPlotZeta}.
Let $a$ be a smooth function on $\gamma$.
Multiplying equation \eqref{SpectralEqn} by $a(s)\zeta_\eps(r)$ and integrating by parts yield
\begin{equation}\label{IntUepsDg}
  \int_{\gamma_\eps} u_\eps a \,d\gamma=\int_{\omega_{\eps,\beta}} (W-\lambda^\eps)u_\eps a\zeta_\eps\,dx-\int_{\omega_{\eps,\beta}} u_\eps \Delta (a\zeta_\eps)\,dx,
\end{equation}
since $\zeta_\eps(\eps)=0$ and $\zeta_\eps'(\eps+0)=1$. Here $\omega_{\eps,\beta}=\{x(s,r)\colon s\in S,\; \eps<r<\beta\}$
is the support of $a\zeta_\eps$.
Similarly, from \eqref{EqnForUoutside} we  obtain the equality
 \begin{equation*}
  \int_{\gamma} u^+ a \,d\gamma=\int_{\omega_{0,\beta}} (W-\lambda)u a\zeta_0\,dx-\int_{\omega_{0,\beta}} u \Delta (a\zeta_0)\,dx,
\end{equation*}
where $\zeta_0(r)=r\zeta(r)$.
It is evident that
\begin{equation*}
   \int_{\omega_{\eps,\beta}} (W-\lambda^\eps)u_\eps a\zeta_\eps\,dx\to
   \int_{\omega_{0,\beta}} (W-\lambda)u a\zeta_0\,dx,
\end{equation*}
 because $\zeta_\eps$ converges to $\zeta_0$ uniformly on $\Real$. Next, we have
\begin{equation*}
  \int_{\omega_{\eps,\beta}} u_\eps \Delta (a\zeta_\eps)\,dx=\int_{\omega_{\frac\beta2,\beta}} u_\eps \Delta (a\zeta_\eps)\,dx+\int_{\omega_{\eps,\frac\beta2}} u_\eps \Delta (a\zeta_\eps)\,dx.
\end{equation*}
The first integral of the right hand side converges to
\begin{equation*}
  \int_{\omega_{\frac\beta2,\beta}} u \Delta (a\zeta_0)\,dx,
\end{equation*}
since $\Delta (a\zeta_\eps)\to \Delta (a\zeta_0)$ uniformly on $[\frac\beta2,\beta]$.
Recalling \eqref{GradientLaplasianInSN}, we can write
\begin{equation*}
  \Delta (a\zeta_\eps)=\left(
  \zeta_\eps\partial_s(a'J^{-1})+ \partial_r(aJ\zeta_\eps')
  \right)\\
  =J^{-1}\left(
  (r-\eps)\partial_s(a'J^{-1})-a\kappa\right)
\end{equation*}
in the set $\omega_{\eps,\frac\beta2}$, since $\zeta_\eps(r)=r-\eps$ for $r\in[\eps,\frac\beta2]$. From this we  conclude that
\begin{multline*}
  \int_{\omega_{\eps,\frac\beta2}}u_\eps \Delta (a\zeta_\eps)\,dx=
 \int_\eps^\frac\beta2 \int_{S} u_\eps(s,r)\big((r-\eps)\partial_s(a'(s)J^{-1}(s,r))
 -a(s)\kappa(s)\big)\,ds\,dr\\
 \to \int_0^\frac\beta2 \int_{S} u(s,r)\big(r\partial_s(a'(s)J^{-1}(s,r))
 -a(s)\kappa(s)\big)\,ds\,dr=\int_{\omega_{0,\frac\beta2}}u \Delta (a\zeta_0)\,dx
\end{multline*}
as $\cE\ni\eps\to 0$, and finally  that
\begin{equation}\label{IntDelta}
  \int_{\omega_{\eps,\beta}} u_\eps\Delta (a\zeta_\eps)\,dx\to\int_{\omega_{0,\beta}} u \Delta (a\zeta_0)\,dx.
\end{equation}

Combining now \eqref{IntUepsDg}--\eqref{IntDelta} we at last deduce
$\int_{\gamma_\eps} u_\eps a \,d\gamma\to \int_{\gamma} u^+ a \,d\gamma$ for all $a\in C^\infty(\gamma)$, hence   $u_\eps(\,\cdot\,,\eps)\to u^+$ in $L^2(S)$ weakly. The proof of the weak convergence for   $u_\eps(\,\cdot\,,-\eps)$ is similar.
\end{proof}

\subsection{Proof of  Theorem~\ref{MainThrmRes}}
Assume first that  the operator $-\frac{d^2}{d r^2}+V$ possesses a zero-energy resonance.
Let $\Psi_\theta$ be the class of functions $\psi$ of compact support that are twice differentiable in $\Real^2\setminus \gamma$,   bounded together with  their first and second derivatives in the closure of  $\Omega^+$ and $\Omega^-$ and $\psi^+=\theta \psi^-$ on $\gamma$.  We also set
\begin{equation*}
\Phi=\{\phi\in W_2^1(\Real^2) \colon \phi \text{ has a compact support}\}.
\end{equation*}
If $\lambda$ and $u$ are the eigenvalue and the corresponding eigenfunction of  $\cH$, then
\begin{equation}\label{IdentityU}
   \int_{\Omega^+}\nabla u \nabla \psi\,dx+\int_{\Omega^-}\nabla u \nabla \psi\,dx
   +\int_{\Real^2}(W-\lambda)u\psi\,dx
   +\int_\gamma \Upsilon u^-\psi^-\,d\gamma=0
\end{equation}
for all $\psi\in \Psi_\theta$, where $\Upsilon=\tfrac12(\theta^2-1)\kappa+\mu$.
We want to take the limit as $\cE\ni\eps\to 0$ in the identity
\begin{equation}\label{IdentityUeps}
   \int_{\Real^2}\big(\nabla u_\eps \nabla \phi+
              (W+V_\eps-\lambda^\eps)u_\eps \phi\big)\,dx=0, \qquad \phi\in \Phi,
\end{equation}
and to obtain \eqref{IdentityU} for the limiting function $u$. But
identities \eqref{IdentityUeps} and \eqref{IdentityU} hold for the different sets of test functions. If $\theta\neq 0$, the set $\Psi_\theta$ is not contained in  $\Phi$, because the functions from $\Psi_\theta$ have   jump discontinuities on $\gamma$.

We introduce the family of operators $R_\eps\colon \Psi_\theta\to \Psi_0$ as follows.
 Let
  $h_1=h_1(n)$ and $h_2=h_2(s,n)$ be  solutions of the Cauchy problems
\begin{align}\label{ProblemH1}
&-h_1''+Vh_1=0,\quad  h_1(-1)=0, \;\; h_1'(-1)=1;
\\\label{ProblemH2}
& -h_2''+Vh_2=\kappa h'+U h,\quad
h_2(s,-1)=0, \;\; \partial_n h_2(s,-1)=0
\end{align}
on the interval $\cI$,  where $h$ is a half-bound state of $-\frac{d^2}{d r^2}+V$ such that $h(-1)=1$.
Given $\psi\in \Psi_\theta$, we write
\begin{equation}\label{Psi0Psi1}
  \psi_0^\eps(s,n)=\psi(s,-\eps)\,h(n), \quad
  \psi_1^\eps(s,n)=\partial_r\psi(s,-\eps)\,h_1(n)
     -\psi(s,-\eps)\,h_2(s,n).
\end{equation}
Then we set
\begin{equation*}
  \hat{\psi}_\eps(x)=
  \begin{cases}
    \psi(x), & \text{if } x\in \Real^2\setminus \omega_\eps,\\
     \psi_0^\eps(s,\tfrac{r}{\eps})
     +\eps \psi_1^\eps\left(s,\tfrac{r}{\eps}\right)& \text{if } x\in \omega_\eps.
  \end{cases}
\end{equation*}
The function $\hat{\psi}_\eps$ is continuous on $\gamma_{-\eps}$ by construction. But it does not in general belong to $W_2^1(\Real^2)$, because it  has a  discontinuity on $\gamma_{\eps}$. Let $R_\eps\psi=\hat{\psi}_\eps+\rho_\eps$, where
\begin{equation*}
\rho_\eps(x)=
\begin{cases}
  -[\hat{\psi}_\eps]_{\eps}\,\zeta(r-\eps),& \text{if }x\in
  \omega_{2\beta}\setminus \omega_\eps,\\
  \phantom{-}0,& \text{otherwise}.
\end{cases}
\end{equation*}
The direct calculations show that $[R_\eps\psi]_{\eps}=0$  and, therefore, $R_\eps\psi$ belongs to $W_2^1(\Real^2)$.
We immediately see that $R_\eps\psi\to \psi$ in $L^2(\Real^2)$ as $\eps\to 0$,
since $\psi_0^\eps$ and $\psi_1^\eps$ are bounded in the small set $\omega_\eps$ and
\begin{multline}\label{JumpHatPsi}
[\hat{\psi}_\eps]_{\eps}(s)=\psi(s,\eps)-\psi_0^\eps(s,1)-\eps\psi_1^\eps(s,1) =\psi(s,\eps)-\theta\psi(s,-\eps)+O(\eps)\\
  =
  \psi(s,+0)-\theta\psi(s,-0)+O(\eps)=O(\eps)
\end{multline}
as  $\eps\to 0$ uniformly in $S$.

\begin{prop}\label{PropIntomegaEps}
For any $\psi\in \Psi_\theta$, we have
  \begin{align}\label{IntInLocal1}
   &\int_Q (\partial_n u_\eps \partial_n \psi_0^\eps+Vu_\eps \psi_0^\eps)J_\eps\,dn\,ds=\eps \int_Q \kappa u_\eps \partial_n\psi_0^\eps\,dn\,ds,
    \\
   &\begin{aligned}\label{IntInLocal2}
    \int_Q (\partial_n u_\eps\partial_n \psi_1^\eps&+Vu_\eps \psi_1^\eps+Uu_\eps \psi_0^\eps)J_\eps\,dn\,ds
    \\
    &=\int_S \Big(u_\eps(s,-\eps)\partial_r \psi(s,-\eps)\big(1+\eps \kappa(s)\big)
    \\
    &-
    \theta^{-1}u_\eps(s,\eps)\big(\partial_r\psi(s,-\eps)-\Upsilon(s)\psi(s,-\eps)\big)(1-\eps \kappa(s))\Big)\,ds
    \\
&-\int_Q \kappa u_\eps\, \partial_n\psi_0^\eps\,dn\,ds+
\eps\int_Q \kappa u_\eps (\kappa \partial_n\psi_0^\eps- \partial_n \psi_1^\eps)\,dn\,ds,
    \end{aligned}
  \end{align}
where $J_\eps(s,n)=1-\eps n \kappa(s)$ and $\psi_k^\eps$ are given by \eqref{Psi0Psi1}.
\end{prop}
\begin{proof}
 The function $\psi_0^\eps$ solves  $-\partial_n^2v+Vv=0$ in $Q=S\times \cI$ and satisfies the conditions $h'(-1)=h'(1)=0$. Then
\begin{multline}\label{ReasAs}
  0=\int_Q u_\eps(-\partial_n^2\psi_0^\eps+V \psi_0^\eps)J_\eps\,dn\,ds
  =-\int_S\psi(s,-\eps) ( u_\eps J_\eps h')\big|_{n=-1}^{n=1}\,ds
  \\+
  \int_Q (\partial_n u_\eps \, \partial_n \psi_0^\eps+Vu_\eps \psi_0^\eps)J_\eps\,dn\,ds+
  \int_Q u_\eps\,\partial_n J_\eps \,\partial_n \psi_0^\eps\,dn\,ds
  \\=\int_Q (\partial_n u_\eps \, \partial_n \psi_0^\eps+Vu_\eps \psi_0^\eps)J_\eps\,dn\,ds-\eps
  \int_Q \kappa u_\eps\,\partial_n \psi_0^\eps\,dn\,ds,
\end{multline}
from which \eqref{IntInLocal1} follows.
Since $h(1)=\theta$, the Lagrange identity $(h_1h'-h_1'h)|_{-1}^1=0$ for  \eqref{ProblemH1} implies
\begin{equation}\label{h1At1Theta}
  h_1'(1)=\theta^{-1}.
\end{equation}
Multiplying the equation in \eqref{ProblemH2} by $h$ and  integrating by parts twice yield
\begin{equation*}
 (h'h_2-h\,\partial_n h_2)\big|_{-1}^1=\kappa(s)\int_{\cI}hh'\,dn
  +\int_{\cI}U(s,n)h^2(n)\, dn.
\end{equation*}
Recalling now \eqref{IntHHpr}, we derive that $\theta \,\partial_n h_2(s,1)=-\tfrac{1}{2}(\theta^2-1)\kappa(s)-\mu(s)$ and finally that
\begin{equation}\label{pdH2IsUpsilon}
\partial_n h_2(s,1)=-\theta^{-1}\Upsilon(s).
\end{equation}
Next, $\psi_1^\eps$ is a solution of $-\partial_n^2v+Vv=-\kappa \partial_n \psi_0^\eps-U\psi_0^\eps$,
which follows from \eqref{ProblemH1} and \eqref{ProblemH2}. Hence
\begin{equation}\label{IntIdenForPsi1}
  \int_Qu_\eps (-\partial_n^2 \psi_1^\eps+V \psi_1^\eps+U \psi_0^\eps)J_\eps\,dn\,ds=-\int_Q \kappa u_\eps\, \partial_n\psi_0^\eps J_\eps\,dn\,ds.
\end{equation}
On the other hand,  integrating by parts with respect to $n$,  we find
\begin{multline*}
  -\int_Q u_\eps \,\partial_n^2 \psi_1^\eps J_\eps\,dn\,ds
  =
  \int_Q (J_\eps\partial_n u_\eps-
  \eps \kappa u_\eps )\partial_n \psi_1^\eps\,dn\,ds
  \\
  -
  \int_Su_\eps(s,\eps n) J_\eps(s,n)\big(\partial_r\psi(s,-\eps)\,h_1'(n)
     -\psi(s,-\eps)\,\partial_n h_2(s,n)\big) \Big|_{n=-1}^{n=1}\,ds
  \\
  =
  \int_Q (J_\eps\partial_n u_\eps-
  \eps \kappa u_\eps )\partial_n \psi_1^\eps\,dn\,ds
  \\-
  \int_S \bigg(\theta^{-1}u_\eps(s,\eps)\big(\partial_r\psi(s,-\eps)
    -\Upsilon(s)\psi(s,-\eps)\big)(1-\eps \kappa(s))
    \\
   -u_\eps(s,-\eps)\partial_r \psi(s,-\eps)\big(1+\eps \kappa(s)\big)
\Big)\,ds,
\end{multline*}
in view of initial conditions \eqref{ProblemH1}, \eqref{ProblemH2} and equalities \eqref{h1At1Theta},  \eqref{pdH2IsUpsilon}. Substituting
the last equality into \eqref{IntIdenForPsi1}, we obtain \eqref{IntInLocal2}.
\end{proof}

\begin{lem}\label{LemmaGradUepsConv}
Under the assumptions of Lemma~\ref{LemmaDconv}, we have
  \begin{align}\label{Assert1}
    &\int_{\Real^2\setminus\omega_\eps}\nabla u_\eps \nabla (R_\eps\psi)\,dx\to
     \int_{\Omega^+}\nabla u \nabla \psi\,dx+\int_{\Omega^-}\nabla u \nabla \psi\,dx,
     \\\label{Assert2}
  &\int_{\omega_\eps}\big(\nabla u_\eps \nabla (R_\eps\psi)+  V_\eps u_\eps R_\eps\psi\big)\,dx\to \int_\gamma \Upsilon u^- \psi^-\,d\gamma
  \end{align}
  as $\cE\ni\eps\to 0$ for all $\psi\in \Psi_\theta$.
\end{lem}
\begin{proof}
Set $\psi_\eps=R_\eps \psi$.
Recalling \eqref{GradientLaplasianInSN}, we write
\begin{multline*}
  \int_{\Real^2\setminus\omega_\eps}\nabla u_\eps \nabla \psi_\eps\,dx=
  \int_{\Omega^+}\chi_\eps\nabla u_\eps \nabla \psi\,dx
  +\int_{\Omega^-}\chi_\eps\nabla u_\eps \nabla \psi\,dx\\
  -\int_\eps^{2\beta}\int_S [\hat{\psi}_\eps]_{\eps}\partial_ru_\eps \zeta'(r-\eps)J\,ds\,dr-
  \int_\eps^{2\beta}\int_S \zeta(r-\eps)\partial_su_\eps\partial_s[\hat{\psi}_\eps]_{\eps} J^{-1}\,ds\,dr.
\end{multline*}
Hence  assertion \eqref{Assert1} follows from  Lemma~\ref{LemmaDconv} \textit{(ii)} and \eqref{JumpHatPsi}.
Next, we have
 \begin{multline}\label{NeedForTh2}
     \int_{\omega_\eps}\big(\nabla u_\eps \nabla \psi_\eps+  V_\eps u_\eps \psi_\eps\big)\,dx
     =\eps^{-1}\int_Q (\partial_n u_\eps \partial_n \psi_0^\eps+Vu_\eps \psi_0^\eps)J_\eps\,dn\,ds
     \\+\int_Q  (\partial_n u_\eps\partial_n \psi_1^\eps+Vu_\eps \psi_1^\eps+Uu_\eps \psi_0^\eps)J_\eps\,dn\,ds
     +\eps\int_Q U u_\eps \psi_1^\eps J_\eps\,dn\,ds
     \\+\eps^2\int_Q \partial_s u_\eps \,\partial_s \psi_\eps J_\eps^{-1} \,dn\,ds.
  \end{multline}
 In view of Proposition~\ref{PropIntomegaEps}, we deduce
  \begin{multline}\allowdisplaybreaks
  \label{IntOmegaEpsTo}
     \int_{\omega_\eps}\big(\nabla u_\eps \nabla \psi_\eps+  V_\eps u_\eps \psi_\eps\big)\,dx
     = \int_S \Big(u_\eps(s,-\eps)\partial_r \psi(s,-\eps)\big(1+\eps \kappa(s)\big)
    \\
    -
    \theta^{-1}u_\eps(s,\eps)\big(\partial_r\psi(s,-\eps)-\Upsilon(s)\psi(s,-\eps)\big)(1-\eps \kappa(s))\Big)\,ds
     \\
     +\eps\int_Q u_\eps\big(\kappa^2\partial_n \psi_0^\eps-\kappa\partial_n \psi_1^\eps+U \psi_1^\eps J_\eps\big)\,dn\,ds
     +\eps^2\int_Q \partial_s u_\eps \,\partial_s \psi_\eps J_\eps^{-1} \,dn\,ds.
  \end{multline}
For any sequence $\{w_\eps\}_{\eps>0}$ bounded in $L^2(Q)$, the estimate
\begin{multline*}
  \left|\int_Q u_\eps(s,\eps n) w_\eps(s,n) \,dn\,ds\right|\leq
  \left(\int_Q |u_\eps(s,\eps n)|^2\,dn\,ds\right)^{1/2}\|w_\eps\|_{L^2(Q)}
  \\
  \leq c_1\left(\eps^{-1}\int_{\omega_\eps} |u_\eps(x)|^2\,dx\right)^{1/2}
  \leq c_2\eps^{-1/2}
\end{multline*}
holds, since $\|u_\eps\|_{L^2(\Real^2)}=1$. Also, we have
\begin{equation*}
 \left| \int_Q \partial_s u_\eps \,\partial_s \psi_\eps J_\eps^{-1} \,dn\,ds\right|=\left| \int_Q u_\eps \,\partial_s(J_\eps^{-1}\partial_s\psi_\eps ) \,dn\,ds\right|\leq c_3\eps^{-1/2},
\end{equation*}
because $\kappa \in C^\infty(\gamma)$ and $\psi\in \Psi_\gamma (\Real^2)$ and, therefore, the function $\partial_s(J_\eps^{-1}\partial_s\psi_\eps )$ is bounded on $Q$ uniformly on $\eps$.
Then \eqref{IntOmegaEpsTo} and Lemma~\ref{LemmaDconv} \textit{(iii)} implies
\begin{multline*}
     \int_{\omega_\eps}\big(\nabla u_\eps \nabla \psi_\eps+  V_\eps u_\eps \psi_\eps\big)\,dx
     \\
     \to \int_S \Big(u(s,-0)\partial_r \psi(s,-0)-
    \theta^{-1}u(s,+0)\big(\partial_r\psi(s,-0)
    -\Upsilon(s)\psi(s,-0)\big)\Big)\,ds
     \\
     =\int_\gamma \Big(u^-\partial_r \psi^--
    \theta^{-1}u^+\big(\partial_r\psi^-
    -\Upsilon\psi^-\big)\Big)\,d\gamma=\int_\gamma\Upsilon u^-\psi^-\,d\gamma,
 \end{multline*}
since $\theta^{-1}u^+ = u^-$.
\end{proof}

Now we can finish the proof of Theorem~\ref{MainThrmRes}.
If $\lambda^\eps\to \lambda$ and $u_\eps\to u$ in $L^2(\Real^2)$ weakly as $\cE\ni\eps\to 0$, then for all $\psi\in \Psi_\theta$
\begin{multline*}
  \int_{\Real^2}\big(\nabla u_\eps \nabla \psi_\eps+
              (W+V_\eps-\lambda^\eps)u_\eps \psi_\eps\big)\,dx
              \\=
               \int_{\Real^2\setminus\omega_\eps}\nabla u_\eps \nabla \psi_\eps\,dx+ \int_{\omega_\eps}\big(\nabla u_\eps \nabla \psi_\eps+
   V_\eps u_\eps \psi_\eps\big)\,dx+\int_{\Real^2}(W-\lambda^\eps)u_\eps \psi_\eps\,dx\\
   \to
     \int_{\Omega^+}\nabla u \nabla \psi\,dx+\int_{\Omega^-}\nabla u \nabla \psi\,dx+\int_\gamma\Upsilon u^-\psi^-\,d\gamma+\int_{\Real^2}(W-\lambda)u\psi\,dx,
\end{multline*}
in view of Lemma~\ref{LemmaGradUepsConv}. Hence the identity \eqref{IdentityU} holds for the pair $(\lambda, u)$. If the limit function $u$ is different from zero, then it must be an eigenfunction of the operator $\cH$ associated with the eigenvalue $\lambda$. If $\lambda\not\in \sigma(\cH)$, then $u=0$.
In view of Lemma~\ref{LemmaQuasimodesHeps} and Proposition~\ref{LemQuasimodes}, there exists an eigenvalue $\lambda^\eps$ of $H_\eps$ such that
\begin{equation*}
  |\lambda^\eps-\lambda|\leq c_4\eps
\end{equation*}
for all $\eps$  small enough, not only for $\eps\in\cE$.

\subsection{Proof of  Theorems~\ref{MainThrmNoRes}}
Suppose now that  the operator $-\frac{d^2}{d r^2}+V$ has no zero-energy resonance.
First we note that the assertions of Lemmas~\ref{LemmaQuasimodesHeps} and \ref{LemmaDconv} are independent of whether  $-\frac{d^2}{d r^2}+V$ has a zero-energy resonance or not.

Set $\Phi_\gamma=\{\phi\in \Phi\colon \phi=0 \text{ on }\gamma\}$.
If $\lambda$ is an eigenvalue with  eigenfunction  $u$  of
the direct sum $\mathcal{D}^-\oplus\mathcal{D}^+$, then $u$ belongs to $\Phi_\gamma$ and
\begin{equation}\label{IdentityU0}
   \int_{\Real^2}\big(\nabla u \nabla \phi
   +(W-\lambda)u\phi\big)\,dx=0\qquad \text{for all }\phi\in \Phi_\gamma.
\end{equation}
In this case the proof is much easier, because $\Phi_\gamma$ is a subspace of $\Phi$. Setting $R_\eps=I$ and arguing as in the proof of Theorem~\ref{MainThrmRes}, we can take the limit as $\cE\ni\eps\to 0$ in \eqref{IdentityUeps} for all $\Phi_\gamma$ and obtain identity \eqref{IdentityU0}. It remains to prove that $u=0$ on $\gamma$.

\begin{lem}\label{UpmZero}
  Under the assumptions of Theorem~\ref{MainThrmNoRes}, we have as $\cE\ni\eps\to 0$ that
   \begin{equation*}
   u_\eps(\,\cdot\,,\pm\eps)\to 0\;\;\text{in }L^2(S)\;\; \text{weakly}.
 \end{equation*}
\end{lem}
\begin{proof}
  Let $h_0$ be the solution of the Cauchy problem
  \begin{equation*}
    -h_0''+Vh_0=0 \;\;\text{in } \cI,\quad  h_0(-1)=1, \;\; h_0'(-1)=0.
  \end{equation*}
Set $\theta_0=h_0(1)$ and $\psi_0^\eps(x)s=\psi(s,-\eps)h_0(\tfrac{r}{\eps})$ for some $\psi\in \Psi_{\theta_0}$. Then the function
\begin{equation*}
  \psi_\eps(x)=
  \begin{cases}
    \psi(x)+(\theta_0\psi(s,-\eps)-\psi(s,\eps))\,\zeta(r-\eps), & \text{if } x\in \Real^2\setminus \omega_\eps,\\
     \psi_0^\eps(x)
     & \text{if } x\in \omega_\eps
  \end{cases}
\end{equation*}
 belongs to $\Phi$. Reasoning as in \eqref{ReasAs} we obtain
\begin{multline*}
  \int_Q (\partial_n u_\eps \partial_n \psi_0^\eps+Vu_\eps \psi_0^\eps)J_\eps\,dn\,ds=-h_0'(1)\int_S u_\eps(s,\eps)\psi(s,-\eps)\,ds\\
  +\eps h_0'(1)\int_S \kappa(s)u_\eps(s,\eps)\psi(s,-\eps)\,ds
  +\eps \int_Q \kappa u_\eps \partial_n\psi_0^\eps\,dn\,ds
\end{multline*}
instead of \eqref{IntInLocal1} in the resonant case.
 From \eqref{NeedForTh2}  we deduce
\begin{equation*}
  \int_{\omega_\eps}\big(\nabla u_\eps \nabla \psi_0^\eps+  V_\eps u_\eps \psi_0^\eps\big)\,dx
     =-\eps^{-1}h_0'(1) \int_S u_\eps(s,\eps)\psi(s,-\eps)\,ds+o(1),
\end{equation*}
as $\eps\to 0$. In the non-resonant case, $h_0'(1)$ is always different from zero. However, from identity \eqref{IdentityUeps} and Lemma~\ref{LemmaDconv} \textit{(ii)} it follows immediately that
\begin{equation*}
  \left|\int_{\omega_\eps}\big(\nabla u_\eps \nabla \psi_0^\eps+  V_\eps u_\eps \psi_0^\eps\big)\,dx\right|\leq C
\end{equation*}
for all $\eps\in\cE$. Hence
\begin{equation*}
  \left|\int_S u_\eps(s,\eps)\psi(s,-\eps)\,ds\right|\leq c\eps,
\end{equation*}
where $c$ does not depend of $\eps$. By the arbitrariness of $\psi$, we have that $u_\eps(\,\cdot\,,\eps)\to 0$ in $L^2(S)$ weakly.
To prove the weak convergence of $u_\eps(\,\cdot\,,-\eps)$ to zero, we can choose $h_0$ as a solution of $-h_0''+Vh_0=0$ in  $\cI$, $h_0(1)=1$,  $h_0'(1)=0$.
\end{proof}

Using Lemma~\ref{UpmZero} and part \textit{(iii)} of Lemma~\ref{LemmaDconv} we get $u^\pm=0$.
The rest of the proof runs as before.


\begin{thebibliography}{10}
\bibitem{GreenMoszkowski1965}
Green, I. M.,  Moszkowski, S. A.  Nuclear coupling schemes with a surface delta interaction. Physical Review, 1965, 139(4B), B790.

\bibitem{Lloyd1965}
Lloyd, P.  Pseudo-potential models in the theory of band structure. Proceedings of the Physical Society, 1965, 86(4), 825.

\bibitem{FaesslerPlastino1967}
Faessler, A.,  Plastino, A.  The surface delta interaction in the transuranic nuclei. Zeitschrift f\"{u}r Physik, 1967, 203(4), 333-345.

\bibitem{Blinder1978}
Blinder, S. M.  Modified delta-function potential for hyperfine interactions. Physical Review A, 1978,18(3), 853.

\bibitem{AntoineGesztesyShabani1987}
Antoine, J. P., Gesztesy, F.,  Shabani, J. Exactly solvable models of sphere interactions in quantum mechanics. Journal of Physics A: Mathematical and General,  1987, 20(12), 3687.

\bibitem{Shimada1992}
Shimada, S. I.  The approximation of the Schrödinger operators with penetrable wall potentials in terms of short range Hamiltonians. Journal of Mathematics of Kyoto University, 1992, 32(3), 583-592.

\bibitem{BehrndtExnerHolzmannLotoreichik2017}
Behrndt, J., Exner, P., Holzmann, M.,  Lotoreichik, V.  Approximation of Schr\"{o}dinger operators with $\delta$-interactions supported on hypersurfaces. Mathematische Nachrichten, 2017, 290(8-9), 1215-1248.

\bibitem{Shimada1994-1}
Shimada, S. I.  Low energy scattering with a penetrable wall interaction. Journal of Mathematics of Kyoto University, 1994, 34(1), 95-147.

\bibitem{Shimada1994-2}
Shimada, S. I. The analytic continuation of the scattering kernel associated with the Schrödinger operator with a penetrable wall interaction. Journal of Mathematics of Kyoto University, 1994, 34(1), 171-190.


\bibitem{ExnerFraas2007}
Exner, P.,  Fraas, M. On the dense point and absolutely continuous spectrum for Hamiltonians with concentric $\delta$ shells. Letters in Mathematical Physics, 2007, 82(1), 25-37.

\bibitem{AlbeverioKostenkoMalamudNeidhardt2014}
Albeverio, S., Kostenko, A., Malamud, M.,  Neidhardt, H. Spherical Schrödinger operators with $\delta$-type interactions. Journal of Mathematical Physics, 2013, 54(5), 052103.

\bibitem{ExnerFraas2008}
Exner, P.,  Fraas, M. Interlaced dense point and absolutely continuous spectra for Hamiltonians with concentric-shell singular interactions. In Mathematical Results In Quantum Mechanics, 2008, pp. 48-65.

\bibitem{ExnerFraas2009}
Exner, P., Fraas, M. On geometric perturbations of critical Schrödinger operators with a surface interaction. Journal of Mathematical Physics, 2009, 50(11), 112101.

\bibitem{DittrichExnerKuhn2016}
Dittrich, J., Exner, P., K\"{u}hn, C.,  Pankrashkin, K. On eigenvalue asymptotics for strong $\delta$-interactions supported by surfaces with boundaries. Asymptotic Analysis, 2016, 97(1-2), 1-25.

\bibitem{MantilePosilicanoSini2016}
Mantile, A., Posilicano, A.,  Sini, M. Self-adjoint elliptic operators with boundary conditions on not closed hypersurfaces. Journal of Differential Equations, 2016, 261(1), 1-55.

\bibitem{BehrndtLangerLotoreichik2013}
Behrndt, J., Langer, M.,  Lotoreichik, V. Schr\"{o}dinger operators with $\delta$- and $\delta'$-potentials supported on hypersurfaces. In Annales Henri Poincaré, 2013, Vol. 14, No. 2, pp. 385-423.

\bibitem{ExnerJex2014}
Exner, P., Jex, M. Spectral asymptotics of a strong $\delta'$ interaction supported by a surface. Physics Letters A, 2014, 378(30-31), 2091-2095.

\bibitem{BehrndtExnerLotoreichik2014-1}
Behrndt, J., Exner, P.,  Lotoreichik, V. Schr\"{o}dinger operators with $\delta$- and $\delta'$-interactions on Lipschitz surfaces and chromatic numbers of associated partitions. Reviews in mathematical physics, 2014, 26(08), 1450015.

\bibitem{Lotoreichik2019}
Lotoreichik, V. Spectral isoperimetric inequalities for singular interactions on open arcs. Applicable Analysis, 2019, 98(8), 1451-1460.

\bibitem{LotoreichikRohleder2015}
Lotoreichik, V.,  Rohleder, J. An eigenvalue inequality for Schr\"{o}dinger operators with $\delta$- and $\delta'$-interactions supported on hypersurfaces. In Operator Algebras and Mathematical Physics, 2015, pp. 173-184.

\bibitem{BehrndtExnerLotoreichik2014-2}
Behrndt, J., Exner, P. ,  Lotoreichik, V. Schr\"{o}dinger operators with $\delta$-interactions supported on conical surfaces. Journal of Physics A: Mathematical and Theoretical, 2014, 47(35), 355202.

\bibitem{OurmieresBonafosPankrashkin2018}
Ourmi\`{e}res-Bonafos, T.,  Pankrashkin, K. Discrete spectrum of interactions concentrated near conical surfaces. Applicable Analysis, 2018, 97(9), 1628-1649.

\bibitem{ExnerRohleder2016}
Exner, P.,  Rohleder, J. Generalized interactions supported on hypersurfaces. Journal of Mathematical Physics, 2016, 57(4), 041507.

\bibitem{ExnerKhrabustovskyi2015}
Exner, P.,  Khrabustovskyi, A. On the spectrum of narrow Neumann waveguide with periodically distributed traps. Journal of Physics A: Mathematical and Theoretical, 2015, 48(31), 315301.

\bibitem{Jex2015}
Jex, M. Spectral asymptotics for a $\delta'$ interaction supported by an infinite curve. In Mathematical Results in Quantum Mechanics: Proceedings of the QMath12 Conference, 2015, (pp. 259-265).

\bibitem{JexLotoreichik2016}
Jex, M.,  Lotoreichik, V.. On absence of bound states for weakly attractive $\delta'$-interactions supported on non-closed curves in $\Real^2$. Journal of Mathematical Physics, 2016, 57(2), 022101.

\bibitem{GolovatyManko2009} Yu.~Golovaty,  S.~Man'ko,
    Solvable models for the Schr\"odinger operators
    with $\delta'$-like potentials     { Ukr. Math. Bulletin} {6} (2) (2009), 169--203;
    arXiv:0909.1034v2 [math.SP].

\bibitem{GolovatyHrynivJPA2010}
    Yu. D. Golovaty,  R. O. Hryniv.
    \textit{On norm resolvent convergence of Schr\"{o}dinger
    operators with $\delta'$-like potentials.}
   Journal of Physics A: Mathematical and Theoretical \textbf{43} (2010) 155204 (14pp) (A Corrigendum: 2011 J. Phys. A: Math. Theor. \textbf{44} 049802)

\bibitem{Golovaty2012}
    Yu.~Golovaty.
    Schr\"{o}dinger operators with  $(\alpha\delta'+\beta \delta)$-like potentials: norm resolvent convergence and solvable models. Methods of Funct. Anal. Topology (3) \textbf{18} (2012), 243--255.

\bibitem{GolovatyHrynivProcEdinburgh2013} Yu. D. Golovaty and R. O. Hryniv. Norm resolvent convergence of singularly scaled Schr\"{o}dinger operators and $\delta'$-potentials. Proceedings of the Royal Society of Edinburgh: Section A Mathematics \textbf{143} (2013),  791-816.

\bibitem{GolovatyIEOT2013}
    Yu. Golovaty. 1D Schr\"{o}dinger operators with short range interactions: two-scale regularization of distributional potentials. Integral Equations and Operator Theory \textbf{75}(3) (2013),   341-362.

\bibitem{ChristianZolotarIermak03}
     P.~L.~Christiansen,  H.~C.~Arnbak,  A.~V.~Zolotaryuk,  V.~N.~Ermakov, Y.~B.~Gaididei.
 On the existence of resonances in the transmission probability for interactions  arising from derivatives of Dirac’s delta function.
     J. Phys. A: Math. Gen. {\bf 36} (2003),     7589--7600.

\bibitem{Zolotaryuk08}
    A. V. Zolotaryuk.
    Two-parametric resonant tunneling across the $\delta'(x)$ potential.
    Adv. Sci. Lett. {\bf 1} (2008), 187-191.

\bibitem{Zolotaryuk09}
    A. V. Zolotaryuk.
    Point interactions of the dipole type defined through a three-parametric power regularization.
    Journal of Physics A: Mathematical and Theoretical {\bf 43} (2010), 105302.

\bibitem{UnverdiTryggvason1992}
Unverdi, S. O., Tryggvason, G. A front-tracking method for viscous, incompressible, multi-fluid flows. Journal of Computational Physics, 1992, 100, 25--37.

\bibitem{JuricTryggvason1996}
Juric, D.,  Tryggvason, G. A front-tracking method for dendritic solidification. Journal of Computational Physics, 1996, 123(1), 127-148.

\bibitem{UddinSung2012}
Uddin, E.,  Sung, H. J. Simulation of flow-flexible body interactions with large deformation. International Journal for Numerical Methods in Fluids, 2012, 70(9), 1089-1102.

\bibitem{Lange2012}
Lange, R. J. Potential theory, path integrals and the Laplacian of the indicator. Journal of High Energy Physics, 2012(11), 32.

\bibitem{GolLavr2000}
Golovaty, Y. D.,  Lavrenyuk, A. S.  Asymptotic expansions of local eigenvibrations for plate with density perturbed in neighbourhood of one-dimensional manifold. Mat. Stud, 2000, 13(1), 51-62.

\bibitem{GolGomLoboPer2004}
 Golovaty, Y. D., G\'{o}mez, D., Lobo, M.,  P\'{e}rez, E. On vibrating membranes with very heavy thin inclusions. Mathematical Models and Methods in Applied Sciences, 2004, 14(07), 987-1034.

\bibitem{GomNazarovPer2020}
G\'{o}mez D., Nazarov S.A., P\'{e}rez-Mart\'{\i}nez M. A Dirichlet Spectral Problem in Domains Surrounded by Thin Stiff and Heavy Bands. In: Constanda C. (eds) Computational and Analytic Methods in Science and Engineering, 2020.

\bibitem{PDEVinitiSpringer}
Fedoruyk M.V., Babich V.M., Lazutkin, V.F., \dots \& Vainberg, B. R.
(1999). Partial Differential Equations V: Asymptotic Methods for
Partial Differential Equations (Vol. 5). Springer Science \&
Business Media.

\end{thebibliography}
\end{document}